\documentclass[12pt]{amsart}
\usepackage[margin=1in]{geometry}
\usepackage{amsmath,amsfonts,amsthm,amssymb,bbm}
\usepackage{graphicx,color,dsfont}
\usepackage[shortlabels]{enumitem}
\usepackage{fourier}
\usepackage{comment}

\usepackage{tcolorbox}

\DeclareMathOperator{\spn}{span}
\DeclareMathOperator{\gKer}{gker}

\DeclareMathOperator{\supp}{supp}

\begin{document}

\newtheorem{theorem}{Theorem}
\newtheorem{lemma}{Lemma}
\newtheorem{proposition}{Proposition}
\newtheorem{example}{Example}
\newtheorem{exercise}{Exercise}
\newtheorem{definition}{Definition}
\newtheorem{corollary}{Corollary}
\newtheorem{notation}{Notation}
\newtheorem{claim}{Claim}

\newtheorem{dif}{Definition}

 \newtheorem{thm}{Theorem}[section]
 \newtheorem{cor}[thm]{Corollary}
 \newtheorem{lem}[thm]{Lemma}
 \newtheorem{prop}[thm]{Proposition}
 \theoremstyle{definition}
 \newtheorem{defn}[thm]{Definition}
 \theoremstyle{remark}
 \newtheorem{rem}[thm]{Remark}
 \newtheorem*{ex}{Example}
 \numberwithin{equation}{section}
\newtheorem*{remark*}{Remark}

\newcommand{\vertiii}[1]{{\left\vert\kern-0.25ex\left\vert\kern-0.25ex\left\vert #1
    \right\vert\kern-0.25ex\right\vert\kern-0.25ex\right\vert}}

\newcommand{\R}{{\mathbb R}}
\newcommand{\C}{{\mathbb C}}
\newcommand{\U}{{\mathcal U}}
\newcommand{\norm}[1]{\left\|#1\right\|}
\renewcommand{\(}{\left(}
\renewcommand{\)}{\right)}
\renewcommand{\[}{\left[}
\renewcommand{\]}{\right]}
\newcommand{\f}[2]{\frac{#1}{#2}}
\newcommand{\im}{i}
\newcommand{\cl}{{\mathcal L}}
\newcommand{\ck}{{\mathcal K}}

\newcommand{\al}{\alpha}
\newcommand{\vro}{\varrho}
\newcommand{\be}{\beta}
\newcommand{\wh}[1]{\widehat{#1}}
\newcommand{\ga}{\gamma}
\newcommand{\Ga}{\Gamma}
\newcommand{\de}{\delta}
\newcommand{\ben}{\beta_n}
\newcommand{\De}{\Delta}
\newcommand{\ve}{\varepsilon}
\newcommand{\ze}{\zeta}
\newcommand{\Th}{\Theta}
\newcommand{\ka}{\kappa}
\newcommand{\la}{\lambda}
\newcommand{\laj}{\lambda_j}
\newcommand{\lak}{\lambda_k}
\newcommand{\La}{\Lambda}
\newcommand{\si}{\sigma}
\newcommand{\Si}{\Sigma}
\newcommand{\vp}{\varphi}
\newcommand{\om}{\omega}
\newcommand{\Om}{\Omega}
\newcommand{\ra}{\rightarrow}

\newcommand{\ro}{{\mathbb  R}}
\newcommand{\rn}{{\mathbb  R}^n}
\newcommand{\rd}{{\mathbb  R}^d}
\newcommand{\rmm}{{\mathbb  R}^m}
\newcommand{\rone}{\mathbb  R}
\newcommand{\rtwo}{\mathbb  R^2}
\newcommand{\rthree}{\mathbb  R^3}
\newcommand{\rfour}{\mathbb  R^4}
\newcommand{\ronen}{{\mathbb  R}^{n+1}}
\newcommand{\ku}{\mathbb  u}
\newcommand{\kw}{\mathbb  w}
\newcommand{\kf}{\mathbb  f}
\newcommand{\kz}{\mathbb  z}

\newcommand{\N}{\mathbb  N}

\newcommand{\tn}{\mathbb  T^n}
\newcommand{\tone}{\mathbb  T^1}
\newcommand{\ttwo}{\mathbb  T^2}
\newcommand{\tthree}{\mathbb  T^3}
\newcommand{\tfour}{\mathbb  T^4}

\newcommand{\zn}{\mathbb  Z^n}
\newcommand{\zp}{\mathbb  Z^+}
\newcommand{\zone}{\mathbb  Z^1}
\newcommand{\zz}{\mathbb  Z}
\newcommand{\ztwo}{\mathbb  Z^2}
\newcommand{\zthree}{\mathbb  Z^3}
\newcommand{\zfour}{\mathbb  Z^4}

\newcommand{\hn}{\mathbb  H^n}
\newcommand{\hone}{\mathbb  H^1}
\newcommand{\htwo}{\mathbb  H^2}
\newcommand{\hthree}{\mathbb  H^3}
\newcommand{\hfour}{\mathbb  H^4}

\newcommand{\cone}{\mathbb  C^1}
\newcommand{\ctwo}{\mathbb  C^2}
\newcommand{\cthree}{\mathbb  C^3}
\newcommand{\cfour}{\mathbb  C^4}
\newcommand{\dpr}[2]{\langle #1,#2 \rangle}

\newcommand{\sn}{\mathbb  S^{n-1}}
\newcommand{\sone}{\mathbb  S^1}
\newcommand{\stwo}{\mathbb  S^2}
\newcommand{\sthree}{\mathbb  S^3}
\newcommand{\sfour}{\mathbb  S^4}

\newcommand{\lp}{L^{p}}
\newcommand{\lppr}{L^{p'}}
\newcommand{\lqq}{L^{q}}
\newcommand{\lr}{L^{r}}
\newcommand{\echi}{(1-\chi(x/M))}
\newcommand{\chip}{\chi'(x/M)}

\newcommand{\wlp}{L^{p,\infty}}
\newcommand{\wlq}{L^{q,\infty}}
\newcommand{\wlr}{L^{r,\infty}}
\newcommand{\wlo}{L^{1,\infty}}

\newcommand{\lprn}{L^{p}(\rn)}
\newcommand{\lptn}{L^{p}(\tn)}
\newcommand{\lpzn}{L^{p}(\zn)}
\newcommand{\lpcn}{L^{p}(\cn)}
\newcommand{\lphn}{L^{p}(\cn)}

\newcommand{\lprone}{L^{p}(\rone)}
\newcommand{\lptone}{L^{p}(\tone)}
\newcommand{\lpzone}{L^{p}(\zone)}
\newcommand{\lpcone}{L^{p}(\cone)}
\newcommand{\lphone}{L^{p}(\hone)}

\newcommand{\lqrn}{L^{q}(\rn)}
\newcommand{\lqtn}{L^{q}(\tn)}
\newcommand{\lqzn}{L^{q}(\zn)}
\newcommand{\lqcn}{L^{q}(\cn)}
\newcommand{\lqhn}{L^{q}(\hn)}

\newcommand{\lo}{L^{1}}
\newcommand{\lt}{L^{2}}
\newcommand{\li}{L^{\infty}}
\newcommand{\beqn}{\begin{eqnarray*}}
\newcommand{\eeqn}{\end{eqnarray*}}
\newcommand{\pplus}{P_{Ker[\cl_+]^\perp}}

\newcommand{\co}{C^{1}}
\newcommand{\ci}{\mathcal I}
\newcommand{\coi}{C_0^{\infty}}

\newcommand{\ca}{\mathcal A}
\newcommand{\cs}{\mathcal S}
\newcommand{\cm}{\mathcal M}
\newcommand{\cf}{\mathcal F}
\newcommand{\cb}{\mathcal B}
\newcommand{\ce}{\mathcal E}
\newcommand{\cd}{\mathcal D}
\newcommand{\cn}{\mathbb N}
\newcommand{\cz}{\mathcal Z}
\newcommand{\crr}{\mathbb  R}
\newcommand{\cc}{\mathcal C}
\newcommand{\ch}{\mathcal H}
\newcommand{\cq}{\mathcal Q}
\newcommand{\cp}{\mathcal P}
\newcommand{\cx}{\mathcal X}
\newcommand{\eps}{\epsilon}

\newcommand{\pv}{\textup{p.v.}\,}
\newcommand{\loc}{\textup{loc}}
\newcommand{\intl}{\int\limits}
\newcommand{\iintl}{\iint\limits}
\newcommand{\dint}{\displaystyle\int}
\newcommand{\diint}{\displaystyle\iint}
\newcommand{\dintl}{\displaystyle\intl}
\newcommand{\diintl}{\displaystyle\iintl}
\newcommand{\liml}{\lim\limits}
\newcommand{\suml}{\sum\limits}
\newcommand{\ltwo}{L^{2}}
\newcommand{\supl}{\sup\limits}
\newcommand{\df}{\displaystyle\frac}
\newcommand{\p}{\partial}
\newcommand{\Ar}{\textup{Arg}}
\newcommand{\abssigk}{\widehat{|\si_k|}}
\newcommand{\ed}{(1-\p_x^2)^{-1}}
\newcommand{\tT}{\tilde{T}}
\newcommand{\tV}{\tilde{V}}
\newcommand{\wt}{\widetilde}
\newcommand{\Qvi}{Q_{\nu,i}}
\newcommand{\sjv}{a_{j,\nu}}
\newcommand{\sj}{a_j}
\newcommand{\pvs}{P_\nu^s}
\newcommand{\pva}{P_1^s}
\newcommand{\cjk}{c_{j,k}^{m,s}}
\newcommand{\Bjsnu}{B_{j-s,\nu}}
\newcommand{\Bjs}{B_{j-s}}
\newcommand{\Ly}{\cl_+i^y}
\newcommand{\dd}[1]{\f{\partial}{\partial #1}}
\newcommand{\czz}{Calder\'on-Zygmund}
\newcommand{\chh}{\mathcal H}

\newcommand{\lbl}{\label}
\newcommand{\beq}{\begin{equation}}
\newcommand{\eeq}{\end{equation}}
\newcommand{\beqna}{\begin{eqnarray*}}
\newcommand{\eeqna}{\end{eqnarray*}}
\newcommand{\bp}{\begin{proof}}
\newcommand{\ep}{\end{proof}}
\newcommand{\bprop}{\begin{proposition}}
\newcommand{\eprop}{\end{proposition}}
\newcommand{\bt}{\begin{theorem}}
\newcommand{\et}{\end{theorem}}
\newcommand{\bex}{\begin{Example}}
\newcommand{\eex}{\end{Example}}
\newcommand{\bc}{\begin{corollary}}
\newcommand{\ec}{\end{corollary}}
\newcommand{\bcl}{\begin{claim}}
\newcommand{\ecl}{\end{claim}}
\newcommand{\bl}{\begin{lemma}}
\newcommand{\el}{\end{lemma}}
\newcommand{\dea}{(-\De)^\be}
\newcommand{\naa}{|\nabla|^\be}
\newcommand{\cj}{{\mathcal J}}
\newcommand{\ubb}{{\mathbb  u}}

\title[Traveling waves for fourth order wave equations]{Existence and stability for  traveling waves of fourth order semilinear wave and Schr\"odinger equations}

\author{Vishnu Iyer}
\address{Department of Mathematics,
	University of Alabama - Birmingham, 
	University Hall, Room 4050, 
	1402 10th Avenue South
	Birmingham AL 35294-1241
	 }\email{vpiyer@uab.edu}

\author{Ross Parker}
\address{IDA Center for Communications Research - Princeton,
805 Bunn Dr, 
Princeton, NJ 08540
}\email{r.parker@idaccr.org}
 
\author[Atanas G. Stefanov]{\sc Atanas G. Stefanov}
\address{ Department of Mathematics,
	University of Alabama - Birmingham, 
	University Hall, Room 4005, 
	1402 10th Avenue South
	Birmingham AL 35294-1241
	 }
\email{stefanov@uab.edu}



\date{\today}
 
\begin{abstract}
 We investigate the existence and spectral stability of traveling wave solutions for a class of fourth-order semilinear wave equations, commonly referred to as beam equations. Using variational methods based on a constrained maximization problem, we establish the existence of smooth, exponentially decaying traveling wave profiles for wavespeeds in the interval $(0, \sqrt{2})$. We derive precise spectral properties of the associated linearized operators and prove a Vakhitov-Kolokolov (VK) type stability criterion that completely characterizes spectral stability. Furthermore, we determine the sharp exponential decay rate of the traveling waves and demonstrate that it matches the decay rate of the Green's function for the linearized operator. Our analysis extends to fourth-order nonlinear Schrödinger equations, for which we establish analogous existence and stability results. The theoretical findings are complemented by numerical computations that verify the stability predictions and reveal the transition from unstable to stable regimes as the wavespeed varies. These results provide a comprehensive mathematical framework for understanding wave propagation phenomena in structural mechanics, particularly suspension bridge models.
\
\end{abstract}

\thanks{ Iyer acknowledges partial support  through a graduate research fellowship from NSF-DMS \# 2204788. Stefanov is partially supported by   NSF-DMS \# 2204788.}

\maketitle

\section{INTRODUCTION}

In this paper, we study the following fourth-order $1 + 1$ wave equation:
\begin{equation}
\label{eq:main}
\begin{cases}
u_{tt} + u_{xxxx} + u - \gamma F(|u|^2)u = 0, & (t, x) \in \mathbb{R}_+ \times \mathbb{R} \\
u(0, x) = u_0(x),
\end{cases}
\end{equation}
where $F$ is a polynomial with non-negative coefficients; $F(r) = \sum_{k=1}^N a_k r^k$ with constants $a_k \geq 0$, and $\gamma > 0$ is a parameter to be determined. This equation belongs to a class commonly known as beam equations, the general form of which can be written as
\begin{equation}
\label{eq:generalbeam}
u_{tt} + u_{xxxx} + f(u) = 0
\end{equation}
for some nonlinearity $f$ that is continuous but not necessarily smooth. Under appropriate boundary conditions, this equation possesses a Hamiltonian structure, with the conserved energy
\begin{equation}
\label{eq:beamH}
H(u) = \int \left(\frac{1}{2} v^2 + \frac{1}{2} (u_{xx})^2 + F(u)\right) dx,
\end{equation}
where $v = u_t$ and $F'(u) = f(u)$. 

This class of equations has physical significance in structural mechanics. McKenna and Walter \cite{McKennaWalter} introduced equation \eqref{eq:generalbeam} with the non-smooth nonlinearity $f(u) = u_+ = \max\{u, 0\}$ as a model for traveling waves on a suspension bridge. A smooth approximation using $f(u) = e^u - 1$ was subsequently considered by Chen and McKenna \cite{ChenMcKenna}. Over the past decade, the existence and stability of solutions for various nonlinearities have been extensively studied through both analytical and numerical approaches (see \cite{DM, McKa, LEV, LEVD, DemPan, DemPan2}).

We are particularly interested in traveling pulse solutions, that is, solutions to \eqref{eq:main} of the form $u(x, t) = \phi(x - ct)$ that are localized at both spatial infinities. Any such traveling wave profile $\phi$ must satisfy the ordinary differential equation
\begin{equation}\label{20b}
\phi_{xxxx} + c^2 \phi_{xx} + \phi - \gamma F(\phi^2)\phi = 0, \quad x \in \mathbb{R}.
\end{equation}

{\bf Remark:} 
In the case of a homogeneous nonlinearity, where $F(x) = x^k$ for some $k > 0$, the parameter $\gamma$ in \eqref{eq:main} is superfluous, as the Lagrange multiplier that appears in our variational construction can be scaled away. However, for general polynomial nonlinearities $F$, due to the lack of scaling symmetry, we are forced to account for this quantity, which is intrinsically tied to the solution of our constrained maximization problem.

Beyond establishing the existence of $\phi$ and determining its properties, we analyze the linearized problem to characterize the spectral stability of these waves. We define the self-adjoint operators $L_\pm: H^4(\mathbb{R}) \to L^2(\mathbb{R})$ that naturally arise during the linearization of the PDE:
\begin{align*}
L_- &= \partial_{xxxx} + c^2 \partial_{xx} + 1 - \gamma F(\phi^2), \\
L_+ &= \partial_{xxxx} + c^2 \partial_{xx} + 1 - \gamma F(\phi^2) - 2\gamma F'(\phi^2)\phi^2.
\end{align*}

To study the stability of the traveling wave $\phi(x - ct)$, we consider another solution $u(x, t)$ of \eqref{eq:main} following the ansatz
$
u(x, t) = \phi(x - ct) + v(t, x - ct).
$
Substituting this into \eqref{eq:main}, simplifying, and linearizing by discarding terms of order $O(v^2)$ and higher, we arrive at the equation
$$
v_{tt} - 2cv_{ty} + c^2 v_{yy} + v_{yyyy} + v -\gamma F(\phi^2)v - 2\gamma F'(\phi^2)\phi^2 v = 0,
$$
where $y = x - ct$ is the traveling-frame spatial variable. This equation can be expressed in terms of the linearized operator $L_+$ as
$$
v_{tt} - 2cv_{ty} + L_+ v = 0.
$$
Setting $w := v_t$, we transform this into the first-order system
$$
\begin{pmatrix} v \\ w \end{pmatrix}_t = \begin{pmatrix} 0 & 1 \\ -L_+ & 2c\partial_y \end{pmatrix} \begin{pmatrix} v \\ w \end{pmatrix} = \begin{pmatrix} 0 & 1 \\ -1 & 2c\partial_y \end{pmatrix} \begin{pmatrix} L_+ & 0 \\ 0 & 1 \end{pmatrix} \begin{pmatrix} v \\ w \end{pmatrix}.
$$

Seeking solutions of the form $\begin{pmatrix} v(t, x) \\ w(t, x) \end{pmatrix} = e^{\lambda t} \begin{pmatrix} v(x) \\ w(x) \end{pmatrix}$ leads to the time-independent eigenvalue problem
\begin{equation}
\label{200b}
\begin{pmatrix} 0 & 1 \\ -1 & 2c\partial_y \end{pmatrix} \begin{pmatrix} L_+ & 0 \\ 0 & 1 \end{pmatrix} \begin{pmatrix} v \\ w \end{pmatrix} = \lambda \begin{pmatrix} v \\ w \end{pmatrix}.
\end{equation}

We say that the traveling wave solution $\phi(x - ct)$ is \emph{spectrally stable} in the context of \eqref{eq:main} if \eqref{200b} admits no non-trivial solution $\begin{pmatrix} v \\ w \end{pmatrix} \neq 0$ corresponding to an eigenvalue $\lambda$ with $\text{Re}\,\lambda > 0$.

As an application of our analysis, we also establish the existence and stability of solitary wave solutions for the fourth-order nonlinear Schrödinger (NLS) equation
\begin{equation}\label{14b}
\begin{cases}
i u_t + u_{xxxx} + \mu u_{xx} - |u|^2 u = 0, & x \in \mathbb{R}, \, t\geq 0 \\
u(0, x) = u_0(x),
\end{cases}
\end{equation}
where $\mu>0$. 
\subsection{Main results}

Our first result establishes that, under appropriate conditions on the wavespeed parameter $c$, solutions to the profile equation \eqref{20b} exist and possess desirable spectral properties.

\begin{theorem}
\label{theo:10b}
Let $0 < c < \sqrt{2}$. Then the profile equation \eqref{20b} has a solution $\phi \in H^\infty(\mathbb{R}) = \bigcap_{j=0}^\infty H^s(\mathbb{R})$. Moreover, the linearized operators $L_\pm$ satisfy the following properties:
\begin{itemize}
\item $L_- \geq 0$, $L_-[\phi] = 0$, $\sigma_{a.c.}(L_-) = [1 - \frac{c^4}{4}, +\infty)$, and there exists a spectral gap property: there exists $\delta > 0$ such that
$$
L_-|_{\{\phi\}^\perp} \geq \delta.
$$
That is, for all $h \perp \phi$, $\langle L_- h, h \rangle \geq \delta \|h\|_{L^2}^2$.

\item $L_+$ has exactly one negative eigenvalue, $L_+[\phi'] = 0$, and $\sigma_{a.c.}(L_+) = [1 - \frac{c^4}{4}, +\infty)$.

\item $L_+$ satisfies the spectral gap property. Specifically, if $\psi_0$ is the eigenfunction corresponding to the unique negative eigenvalue, then
$$
L_+|_{\{\psi_0, \ker[L_+]\}^\perp} \geq \delta.
$$
\end{itemize}
\end{theorem}

Our next result characterizes the spectral stability of the waves $\phi$ constructed in Theorem \ref{theo:10b}.

\begin{theorem}
\label{theo:20b}
Let $0 < c < \sqrt{2}$ and let $\phi$ be the wave generated in Theorem \ref{theo:10b}. Assume that the wave $\phi$ is non-degenerate, that is, $\ker(L_+) = \spn[\phi']$. Then the wave $\phi(x - ct)$ is spectrally stable if and only if the Vakhitov-Kolokolov type condition holds:
\begin{equation}
\label{eq:vk}
4c^2 \langle L_+^{-1} \phi'', \phi'' \rangle + \|\phi'\|_{L^2}^2 < 0.
\end{equation}
\end{theorem}

{\bf Remark:}
\begin{enumerate}
\item Note that $\text{span}[\phi'] \subseteq \ker(L_+)$. Thus, the non-degeneracy condition requires that the kernel of $L_+$ is minimal, which should hold generically.

\item If we only have $\text{span}[\phi'] \subseteq \ker(L_+)$, we can still claim stability when \eqref{eq:vk} holds. However, if $\ker(L_+)$ contains additional elements and $4c^2 \langle L_+^{-1} \phi'', \phi'' \rangle + \|\phi'\|_{L^2}^2 > 0$, we cannot definitively conclude instability.

\item The Vakhitov-Kolokolov condition \eqref{eq:vk} admits the usual interpretation in terms of monotonicity properties. Indeed, assuming the mapping $c \mapsto \phi_c$ is Gâteaux differentiable and differentiating with respect to $c$, we can compute $L_+^{-1}[\phi''] = -\frac{1}{2c} \partial_c \phi$, which yields
$$
4c^2 \langle L_+^{-1} \phi'', \phi'' \rangle + \|\phi'\|_{L^2}^2 = c\partial_c \|\phi'\|_{L^2}^2 + \|\phi'\|_{L^2}^2 = \partial_c(c\|\phi_c'\|_{L^2}^2).
$$
Thus, the VK condition \eqref{eq:vk} holds on all open intervals where the function $c \mapsto c\|\phi_c'\|_{L^2}^2$ is decreasing.
\end{enumerate}

Our final result determines the precise exponential decay rate of the traveling wave $\phi$.

\begin{theorem}
\label{theo:30b}
Let $0 < c < \sqrt{2}$ and let $\phi$ be the wave generated in Theorem \ref{theo:10b}. Then $\phi$ and all its derivatives satisfy the exponential bound
$$
|\phi^{(j)}(x)| \leq C_j e^{-\frac{\sqrt{2-c^2}}{2}|x|}, \quad j = 0, 1, 2, \ldots
$$
\end{theorem}

\newpage

\section{Preliminaries}
We start with some basics. 
\subsection{ Elementary inequalities}
Following the inequalities in the Introduction, for $f\in H^2(\rone)$, we have the GNS inequality:
  for any $2<p<\infty$, 
\begin{equation}
	\label{30b} 
	\|f\|_{L^p(\rone)}\leq C_p \|f\|_{\dot{H}^{s_p}(\rone)}\leq C_p \|f'\|_{L^2}^{s_p} \|f\|_{L^2}^{1-s_p}, \ \ s_p=\f{1}{2}-\f{1}{p}.
\end{equation}
It is worth noting that for $p=\infty$, the first inequality above fails, whereas  we still have the penultimate bound 
\begin{equation}
	\label{32b} 
	\|f\|_{L^\infty(\rone)}^2\leq  C \|f'\|_{L^2}  \|f\|_{L^2}. 
\end{equation}
Similarly, there is control in terms of the second derivative, namely 
\begin{equation}\label{33b}
\|f\|_{L^p(\rone)}\leq C_p \|f\|_{\dot{H}^{s_p}(\rone)}\leq C_p \|f''\|_{L^2}^{\f{s_p}{2}} \|f\|_{L^2}^{1-\f{s_p}{2}},
\end{equation}
more precisely,
\begin{equation}
	\label{34b} 
	\|f\|_{L^p(\rone)}^p\leq C_p \|f''\|_{L^2}^{\(\f{p}{4}-\f{1}{2}\right)} \|f\|_{L^2}^{\left(\f{3p}{4}+\f{1}{2}\right)}.
\end{equation}

An elementary inequality, which will be useful in the sequel is as follows: for any $\de_1\neq \de_2, \de_1>0, \de_2>0$, 
\begin{equation}
	\label{324b} 
	\int_{-\infty}^\infty e^{-\de_1|x-y|} e^{-\de_2|y|} dy \leq C e^{-\min(\de_1, \de_2) |x|}. 
\end{equation}
Note that even for $\de=\de_1=\de_2$, one obtains a slightly worse bound $C_\eps e^{-(\de-\eps) |x|}$ for any $\eps>0$. 
\subsection{Concentration compactness.}
We have the following standard concentration compactness result, as found in  \cite{Lions}. 
\begin{theorem}
	\label{ConcCompb}Let $(\rho_{n})$ be a sequence in $L^{1}(\mathbb{R})$
	satisfying:
	$$
	\rho_{n}\geq0\ \text{ and }\int_{\mathbb{R}}\rho_{n}\ dx=\lambda,
	$$
	where $\lambda>0$ is fixed. Then there exist a subsequence $(\rho_{n_{k}})$
	satisfying one of the three following possibilities:
	\begin{enumerate}[(i)]
	\item (compactness) there exist $y_{k}\in\mathbb{R}$ such that $\rho_{n_{k}}(\cdot+y_{k})$
	is tight i.e.:
	$$
	\forall\epsilon>0,\exists R<\infty,\ \int_{y_{k}+B_{R}}\rho_{n_{k}}\ dx\ge\lambda-\epsilon;
	$$
	
	\item (vanishing) $\lim_{k\to\infty}\sup_{y\in\mathbb{R}}\int_{y+B_{R}}\rho_{n_{k}}\ dx=0,$
	for all $R<\infty;$
	
	\item (splitting) there exist $\alpha\in(0,\lambda)$ such that for
	all $\epsilon>0$, there exist $k_{0}\geq1$ and $\rho_{k}^{1}, 
	\rho_{k}^{2}\in L^{1}(\mathbb{R})$ satisfying for $k\geq k_{0}:$
	$$
	\left\|\rho_{n_{k}}-(\rho_{k}^{1}+\rho_{k}^{2})\right\|_{1}<\epsilon,\qquad\left|\int_{\mathbb{R}}\rho_{k}^{1}\ dx-\alpha\right|<\epsilon,\qquad\left|\int_{\mathbb{R}}\rho_{k}^{2}\ dx-(\lambda-\alpha)\right|<\epsilon,
	$$
	and $\textup{dist}(\supp(\rho_{k}^{1}), \supp(\rho_{k}^{2}))\to\infty.$
    \end{enumerate}
\end{theorem}

\subsection{Instability index count}
Here, we present a short introduction to the theory of instability indices, which helps us decide the spectral stability of the waves. For a self-adjoint operator $S$, the Morse index $n(S)$  is taken to be  the dimension of the negative subspace of $S$. For an eigenvalue problem of the form 
\begin{equation}
	\label{e:10} 
	\cj\cl u=\la u,
\end{equation}
with $\cj^*=-\cj, \cl^*=\cl$, we introduce the generalized kernel 
$$
\gKer(\cj\cl)=\spn\{u: (\cj\cl)^l u=0, l=1, 2, \ldots \}. 
$$
Assuming $\dim(\gKer(\cj\cl))<\infty$, we introduce a basis $\{\eta_j\}_{j=1}^N$, and a symmetric matrix $\cd$, with entries 
$$
\cd_{i j}=\cd_{ji}=\dpr{\cl \eta_i}{\eta_j}.
$$
There are  three different type of point spectra arising in the eigenvalue problem \eqref{e:10}, namely unstable real spectrum, denoted $k_r:=\{\la>0: \la\in \si_{p.p.}(\cj\cl)\}$, the unstable complex spectra, $k_c=\{\la: \la\in \si_{p.p.}(\cj\cl), \text{Re }\la>0, \text{Im }\la>0\}$ and finally, the marginally stable spectrum with negative Krein signature, defined by $k_{i}^-=\{i \mu, \mu>0: \cj\cl f=i \mu f, \dpr{\cl f}{f}<0 \}$. The instability index formula then reads 
\begin{equation}
	\label{e:20} 
	k_r+2 k_c+2k_i^-=n(\cl)-n(\cd).
\end{equation}
Note that the number of unstable modes, counted with multiplicities,  is exactly $k_r+2k_c$, whereas right-hand side of \eqref{e:20} is at least as much. In the particular case $n(\cl)=1$, which will be of the main interest herein, the formula \eqref{e:20} yields a rather precise information about the stability for \eqref{e:10}, namely that the stability occurs exactly when $n(\cd)=1$. We formulate this as a corollary. 
\begin{cor}
    \label{cor:po} If $n(\cl)=1$, then the eigenvalue problem \eqref{e:10} has no non-trivial solutions (that is it is spectrally stable) if and only if $n(\cd)=1$. 
\end{cor}

\section{Variational construction: Proof of Theorem \ref{theo:10b}}

We construct our traveling wave solution $\phi$ as the extremizer of a constrained maximization problem.

For $\phi \in H^2(\rone)$, we look to maximize the
quantity, 
$$
I[\phi]:=\int_{\rone}G(|\phi(x)|^2)\ dx,
$$where $G'(r)=F(r)$ or equivalently,
$$
G(r)=\sum_{k=1}^N\frac{a_k}{k+1}r^{k+1},
$$
subject to the constraint, 
$$
L[\phi]:=\int_{\rone}(\phi'')^{2}-c^{2}(\phi')^{2}+\phi^{2}\ dx=\lambda,
$$for any given positive $\la$.

We now move on to establish that this constrained optimization problem is well defined. We start with a few simple lemmas. 
 \subsection{Some preparatory lemmas}
 \begin{lemma}
 	\label{le:10b}
 Let $c\in (0, \sqrt{2})$. Then,  constraint functional $L$ satisfies the coercivity estimate 
 	\begin{equation}
 		\label{40b} 
 		L[f]\geq \de_c \|f\|_{H^2}^2, 
 	\end{equation}
 for some $\de_c>0$. 
 \end{lemma}
 \begin{proof}
 	By Plancherel's,
 	$$
 	L[f]=\int(\xi^{4}-c^{2}\xi^{2}+1)|\hat{f}(\xi)|^{2}d\xi. 
 	$$
 	Then, it is clear by completing the square that 
 	$$
 	\xi^{4}-c^{2}\xi^{2}+1 \geq (1-\f{c^4}{4}) \max(1, \xi^4),
 	$$
 	whence 
 	$$
 	L[f]\geq (1-\f{c^4}{4}) \int \max(1, \xi^4) |\hat{f}(\xi)|^{2}d\xi \geq \de_c \|f\|_{H^2}^2, 
 	$$
 	where one could take $\de_c:=\f{(1-\f{c^4}{4})}{10}>0$. 
 \end{proof}
 
Observe by the Sobolev embedding,  \eqref{34b}, one has for $p>2$
$$
\|\phi\|_{L_p}^p\leq C_p \|\phi''\|_{L^2}^{\(\f{p}{4}-\f{1}{2}\right)} \|\phi\|_{L^2}^{\(\f{3p}{4}+\f{1}{2}\right)}\leq C_p \|\phi\|_{H^2}^p.
$$

This combined with the  definition of the functional $I[\cdot]$, and Lemma \ref{le:10b}, gives, 
$$
I[\phi]=\sum_{k=1}^N\f{a_k}{k+1}\|\phi\|_{2k+2}^{2k+2}\leq \sum_{k=1}^N \f{a_k}{k+1}\|\phi\|_{H^2}^{2k+2}\leq C \sum_{k=1}^N \f{a_k}{k+1}\(L[\phi]\right)^{k+1} =CG(\la) <\infty
$$

Thus, for a function $\phi$ satisfying the constraint $L[\phi]=\la$, we have $I[\phi	]\leq C_\lambda<\infty$. In fact, 
introduce 
$$
0< M_{\lambda}:=\sup_{\phi\in H^{2}(\rone),\ L[\phi]=\lambda}I[\phi]\leq C_\lambda.
$$
 With this, the variational problem 
\begin{equation}
	\label{60b} 
	\left\{ \begin{array}{c}
		I[u]\to \max \\
		L[u]=\la.
	\end{array}
	\right.
\end{equation}
is well-posed for each $\la>0$. 

 We now prove the superlinearity of $M_\lambda$. 
 \begin{lemma}[Superlinearity of $M_{\la}$]
 	\label{le:5b} 
 	For $\la >0$ and $\al\in(0,\lambda)$,
 	$$
 	M_\la > M_\al +M_{\la-\al}.
 	$$
 \end{lemma}
 
 \begin{proof}
 As $G(\cdot)$ is a polynomial with non negative coefficients, for $\beta >1$ and $\phi \in H^2$,  we obtain the simple strict inequality, 
$$
\beta^2 I[\phi] = \int	\beta^2 G[|\phi|^2]\ dx < \int	G[|\beta \phi|^2] \ dx = I[\beta\phi].
$$

So, for $\la>0$ and $\al\in (0,\la)$, suppose $\{\phi_n\}$ is a maximizing sequence for $I$ with constraint $L[\phi_n]=\al$, then by the above inequality, 
$$
\f{\la}{\al}M_\alpha = \lim_{n\to \infty}\f{\la}{\al}I[\phi_n]<\lim_{n\to\infty}I\left[\sqrt{\f{\la}{\al}}\phi_n\right] \leq M_\lambda.
$$
Thus, without loss of generality, assuming $\al\in [\la/2,\la)$, we have, 
$$
M_\lambda >\f{\la}{\al}M_\al=\f{\al +(\la-\al)}{\al}M_\al = M_\al + \f{\la -\al}{\al}M_\al > M_\alpha +M_{\la-\al}.
$$

In fact, by nature of the polynomial $G(x)$ and the functional $I[\cdot]$, one may run the same arguments as above to conclude, for $\beta >1$,
$$
\beta^4I[\phi]<I[\beta \phi],
$$ and whence, for $0<\alpha<\la,$
\begin{equation}\label{12b}
\f{M_\alpha}{\al^2}<\f{M_\la}{\la^2}.
\end{equation}
 \end{proof}
 
 \subsection{The variational problem \eqref{60b} has a solution}
 
 We have the following proposition. 
 \begin{proposition}
 	\label{prop:10b} 
 	Let $\la>0, c\in (0, \sqrt{2})$. Then, the variational problem \eqref{60b} has a solution $U\in H^2(\rone)$.  

 \end{proposition}
\begin{proof}
	We take a maximizing sequence, say $u_k$, so that $L[u_k]=\la$ and $\lim_k I[u_k]=M_\la>0$.  Due to Lemma \ref{le:10b}, this is a bounded sequence in $H^2(\rone)$, that is $\sup_k \|u_k\|_{H^2}= K<\infty$. After passing to a subsequence (maybe several subsequences), we may assume, without loss of generality, that there exist $L_1, L_2, L_3$, all strictly positive, so that 
$$
	\lim_k \|u_{k}\|_{2}^{2}=L_{1}>0,\quad \lim_k \|u'_{k}\|_{2}^{2}= L_{2}>0,\quad \lim_k\|u''_{k}\|_{2}^{2}=L_{3}>0.
$$
	Indeed, if say $L_1=0$, then from \eqref{34b}, 
	\begin{equation}
	\lim_k\|u_k\|_{L^p(\rone)}^p\leq C_p \lim_k \|u_k''\|_{L^2}^{\(\f{p}{4}-\f{1}{2}\right)} \|u_k\|_{L^2}^{\(\f{3p}{4}+\f{1}{2}\right)}=0,
	\end{equation}
	for all $p\geq 2$. Hence, $M_\la=\lim_k I[u_k]=0$, a contradiction. Similar contradiction is reached, if any of $L_2$ or $L_3$ are assumed to be zero. 
	
	Next, there is a sequence $c_k: \lim_k c_k=1$, so that $c_k \|u_k\|^2=L_1$, so that we introduce a new sequence $\rho_k:=c_k u_k^2$, so that $ \int \rho_k(x) dx=L_1$. We apply the concentration compactness theorem, Theorem \ref{ConcCompb}, to the sequence of functions $\{\rho_k\}$ . 
	
	We show that the only possibility for $\rho_k$ is the compactness/tightness. To this end, we need to refute the other two possibilities. 
	
	\subsubsection{Vanishing cannot occur}
	Assume, for a contradiction, that $\{\rho_k\}$ vanishes. Then for every $\epsilon>0$
	there exist $N=N_\eps>0$, such that for all $k>N$, 
	\begin{equation}
		\label{90b}
		\sup_{y\in \rone} \int_{y-2}^{y+2} u_{k}^2(x)\ dx<\epsilon. 
	\end{equation}
	Introduce $\eta\in C^{\infty}$ be such that $0\leq\eta(x)\leq1$ and 
	$
	\eta(x)=\begin{cases}
		1 & |x|\leq1\\
		0 & |x|>2
	\end{cases}.
	$
	
	An easy observation is that  $u_k\eta(\f{\cdot -y}{R}) \in H^2$ with, 
	$$
			\left\|u_k\eta\left(\f{x-y}{R}\right)\right\|_{H^2}\leq C_{\eta,R}\|u_k\|_{H^2}\leq C_{\eta,R}K;
	$$
	or in other words is bounded. 
	
	Applying inequality \eqref{30b}, we obtain for all $p\geq 4,$
	
	\begin{align*}
	\int_{\rone}u_{k}^{p} dx  \leq &  \sum_{j}\int_{j-1}^{j+1} u_{k}^{p}\ dx\leq\sum_{j}\int \left(u_{k}(x)\eta(x-j)\right)^{p}\ dx \\
		  \leq & \  C \sum_{j}  \|u_{k}\eta(\cdot-j)\|_{\dot{H}^1}^{\f{p}{2}-1} \|u_{k}\eta(\cdot-j)\|_{L^2}^{\f{p}{2}+1} \\  
		  \leq & \  C \|u_k\|_{H^2}^{\f{p}{2}-1}  \sup_j \|u_{k}\eta(\cdot-j)\|_{L^2}^{\f{p}{2}-1}\sum_{j}   \|u_{k}\eta(\cdot-j)\|_{L^2}^2\\ 
		  \leq & \ C  \|u_k\|_{H^2}^{\f{p}{2}+1} \sup_j \|u_{k}\eta(\cdot-j)\|_{L^2}^{\f{p}{2}-1}\leq  \ C \eps^{\f{p}{2}-1}, 
	\end{align*}

	where we have used \eqref{90b}, along with the uniform bounds on $\|u_k\|_{H^2}$. 
	Hence it follows that $M_\lambda=\lim_k I[u_k]=\lim_k \int_\rone	G[|u_k|^2]\ dx \leq C\eps$, 
	for all $\epsilon>0$, which is a contradiction as $M_\lambda >0.$
	
	\subsubsection{Splitting cannot occur}
Assume, for a contradiction, that the splitting alternative in Theorem \ref{ConcCompb} holds. 	By passing to a subsequence and translating the functions $\rho_k(\cdot-y_k)$, if necessary, we have that there exists $R_k: \lim_k R_k=+\infty$, so that $\text{supp } \rho_k^1 \subset (-\infty, -R_k)$, while $\text{supp }\rho_k^2 \subset (R_k, +\infty)$. 

In fact, introduce functions $\zeta_\pm\in C^\infty(\rone): 0\leq \zeta_\pm(x)\leq 1$, so that 
\begin{eqnarray*}
	\zeta_+(x) &=& (1-\eta(x)) \chi_{(0, +\infty)}(x)=\begin{cases}
		0& x\leq 1\\
		1 & x>2
	\end{cases}\\
	 \zeta_-(x)& = & (1-\eta(x)) \chi_{(-\infty,0)}(x)=\begin{cases}
		1& x\leq -2\\
		0& x>-1
	\end{cases}.
\end{eqnarray*}
Note that $\zeta_-(x) + \eta(x)+ \zeta_+(x) = 1$. These functions induce a decomposition, $u_k=u_{k,-}+u_{k,0}+u_{k,+}$, where
$$
u_{k,\pm}(x):= u_k(x) \zeta_\pm(x/R_k), u_{k,0}=u_k\eta(x/R_k). 
$$
	Given our assumption for splitting, it follows that for some $\al\in (0, L_1)$ and for all $k\geq k_0$, 
	$$
	\left|c_{k} \int u_{k,-}^{2}(x) \ dx-\alpha\right|<2 \epsilon
	, \ \quad\left|c_k \int u_{k,+}^{2}(x)\ dx-(L_{1}-\alpha)\right|<2\epsilon, \ c_k\int u_{k,0}^{2}(x) dx<2\eps.
	$$
	Dividing by $c_k$ yields, for large enough $k$, 
	\begin{equation}
		\label{110b} 
			\left|\int u_{k,-}^{2}(x) \ dx-\f{\alpha}{c_k}\right|<4 \epsilon
		, \ \quad\left| \int u_{k,+}^{2}(x)\ dx-\f{(L_{1}-\alpha)}{c_k}\right|<4\epsilon, \ \int u_{k,0}^{2}(x) dx<2\eps.
	\end{equation}
	We now analyze the quantity $L[u_k]$. We have 
	\begin{align}
		L[u_{k}] & =L[u_{k,-}]+L[u_{k,+}]+L[u_{k,0}]+\nonumber \\
		\label{100b}
		& \ \quad+2 \left(\int (u_{k,-}''+u_{k,+}'') u_{k,0}''-c^{2}\int (u_{k,-}'+ u_{k,+}')u_{k,0}'+
		\int (u_{k,-}+u_{k,+})u_{k,0}\right). 
	\end{align}
	We deal with the main terms first. We have 
	\begin{eqnarray*}
		\int u_{k,\pm}'' u_{k,0}'' dx = 	\int (u_{k}''(x))^2  \zeta_\pm(x/R_k) \eta(x/R_k)dx +E_k, 
	\end{eqnarray*}
	where in the terms $E_k$, at least one derivative has fallen on the cutoffs, so $|E_k|\leq C R_k^{-1} \|u_k\|_{H^2}^2$. Further, integration by parts yields 
	$$
	\left|\int (u_{k,-}'+ u_{k,+}')u_{k,0}' dx\right|=
	\left|\int (u_{k,-}''+ u_{k,+}'')u_{k,0} dx\right|\leq C \|u_k\|_{H^2}\|u_{k,0}\|_{L^2}\leq C \sqrt{\eps}. 
	$$
	by \eqref{110b}. Similarly, 
	$$
	\left|\int (u_{k,-}+u_{k,+})u_{k,0}\right|\leq C \|u_k\|_{L^2} \|u_{k,0}\|_{L^2}\leq C \sqrt{\eps}. 
	$$
	Putting all these estimates together 
 \begin{align*}
 L[u_{k}]-L[u_{k,-}]-L[u_{k,+}] & \geq \ 
 	L[u_{k,0}]+2 	\int (u_{k}''(x))^2  \zeta_\pm(x/R_k) \eta(x/R_k)dx  -C\sqrt{\epsilon}-C R_k^{-1} \\
 	&\geq \ -C\sqrt{\epsilon} - C R_k^{-1},
 \end{align*}
	thereby proving, 
	\begin{equation}
		\label{115b}
				\liminf_{k\to\infty}\left(L[u_{k}]-L[u_{k,-}]-L[u_{k,+}]\right)\geq  -C\sqrt{\epsilon}. 
	\end{equation}
Recall that $L[u_k]=\la$, and after possibly taking a subsequence, we may and do assume that $\lim_k L[u_{k,\pm}]=\si_\pm$. Note that $\si_\pm>0$. Otherwise, say $\si_+=0$, and then  by Lemma \ref{le:10b}, it would follow that $\lim\|u_{k,+}\|_{H^2}=0$, which contradicts \eqref{110b}. Thus, we have established, following \eqref{115b}, 
\begin{equation}
	\label{120b} 
	\si_-+\si_+\leq \la+ C\sqrt{\epsilon}. 
\end{equation} Here, as $\sigma_-$ and $\sigma_+$ are independent of $\epsilon$, we may further assume that $\sigma_{-},\sigma_+<\la.$

We now analyze $I[u_k]=\int G[|u_k|^2]\ dx=\sum_{j=1}^n\f{a_j}{j+1}\int|u_k|^{2j}\ dx$. Recall $u_k=u_{k,-}+u_{k,+}+u_{k,0}$, where $u_{k,-}, u_{k,+}$ have disjoint support. We have 
\begin{eqnarray*}
	I[u_k]= I[u_{k,-}]+I[u_{k,+}] +\ce_k,
\end{eqnarray*}
where the error terms $\ce_k$ consists of linear combination of  integrals 
with all terms from \\ $(u_{k,-}+u_{k,+}+u_{k,0})^{2j}$ for all $j$, with at least one entry $u_{k,0}$. Applying H\"older's and Sobolev embedding,  we have,
$$
|\ce_k|\leq \sum_{j=1}^n C_j \|u_k\|_{L^\infty}^{2j} \|u_k\|_{L^2} \|u_{k,0}\|_{L^2}\leq C\sqrt{\eps}. 
$$
It follows that 
$$
I[u_k]\leq M_{L[u_{k,-}]}+M_{L[u_{k,+}]}+C\sqrt{\eps}. 
$$
After taking limits in $k$ and by inequality \eqref{12b}
$$
 M_\la\leq M_{\si_-}+M_{\si_+}+C\sqrt{\eps} \leq \f{\si_-^2}{\la^2}M_{\la}+\f{\si_+^2}{\la^2}M_{\la}+C\sqrt{\epsilon},
$$ which upon rearranging and using relations \eqref{120b} yields the bound, 
$$
M_\la \leq \f{\la^2}{\la^2-\si_-^2-\si_+^2}C\sqrt{\epsilon}\leq \left(\f{\la^2}{2\si_-\si_+-2\la C\sqrt{\epsilon}}\right)C\sqrt{\epsilon}.
$$
 This is a contradiction, since $M_\lambda>0$, while $\eps$ was taken to be arbitrary. 
 
 \subsubsection{Tightness.}
 So, it follows that the tightness alternative holds. Relabeling $\{u_k\}$ to be the sequence of functions $u_k: = u_k(x-y_k)$, bounded in $H^2$, after choosing a subsequence, converges weakly to $\tilde{U} \in H^2$. 
 After choosing another subsequence, and by the Rellich's compactness theorem, we have a limiting function, $U: \lim_k\|u_k-U\|_{L^2}=0$, and by uniqueness of weak limits in $L^2,$ we have $U=\tilde{U}\in H^2.$ 
 
Immediately, by Sobolev inequality \eqref{34b}, we see that 
$$
\lim_k\|u_k-U\|_{L^p}^p\leq C_p\lim_k \|(u_k-U)''\|_{L^2}^{\(\f{p}{4}-\f{1}{2}\right)} \|(u_k-U)\|_{L^2}^{\(\f{3p}{4}+\f{1}{2}\right)} = 0.
$$ 
Whence, it follows, $I[U]=\lim_k I[u_k]=M_\la$. 

In fact, by inequality \eqref{33b}, $\lim_k\|u_k-U\|_{\dot{H}^1}=0$, and hence $\lim_k \|u_k-U\|_{H^1}=0$. On the other hand,  by the lower semi-continuity of the $H^2$ norm with respect to weak convergence, we have that $\liminf_k \int_{\rone}(u_k'')^{2}\geq \int_{\rone}(U'')^{2}$. Hence, 
\begin{align*}
	\la=\liminf_k L[u_k] &= \ \liminf_k  \int_{\rone}(u_k'')^{2}-c^{2}(u_k')^{2}+u_k^{2}\ dx \\
	& =\ \liminf_k  \int_{\rone}(u_k'')^{2} -  \int c^{2}(U')^{2}+U^{2}\ dx\\
	&\geq \  \int_{\rone}(U'')^{2} -   c^{2}(U')^{2}+U^{2}\ dx=L[U].
\end{align*}
Thus, $L[U]\leq \la$, while $I[U]=M_\la$. Owing to superlinearity of  $M_\lambda$, this presents a contradiction, unless $L[U]=\la$, and hence $U$ is a maximizer. Consequently, $\lim_k \|u''_k\|_{L^2}=\|U''\|_{L^2}$. By weak convergence of $u_k$ to $U$ in $H^2$ and the fact that $\lim_k\|u_k\|_{H^2}=\|U\|_{H^2}$,  it follows that we can upgrade the convergence to strong in $H^2$, i.e.  $\lim_k \|u_k-U\|_{H^2}=0$. 
\end{proof}

\subsection{Euler-Lagrange equation and spectral properties}
 Now that we know that the variational problem \eqref{60b} has a solution, we derive its  Euler-Lagrange equation, along with some extra variational properties. 
 \begin{proposition}
 	\label{prop:20b} 
 	The maximizer $U$ of \eqref{60b}  satisfies the Euler-Lagrange equation, 
 	 	\begin{equation}
 		 		\label{80b} 
 		 		U''''+c^2 U''+U - \kappa^{-1} F(|U|^2) U=0,
 		 	\end{equation}
 		 	where $\kappa^{1}=\f{1}{\la}\int_\rone F[|U|^2]|U|^2 \ dx.$
 	In addition, $U\in H^\infty(\rone)$. 	For the linearized operators, 
 	 	\begin{eqnarray*}
 	 		\cm_- &=& \p_{xxxx}+c^2 \p_{xx}+1 -  \kappa^{-1} F(U^2), \\
 	 		\cm_+ &=& \p_{xxxx}+c^2 \p_{xx}+1 -\kappa^{-1} F(U^2)- 2\kappa^{-1} F'(U^2) U^2.
 	 	\end{eqnarray*}
  	we have 
  		\begin{itemize}
  		\item $\cm_-\geq 0$, $\cm_-[U]=0$ 
  		$ \si_{a.c.}(\cm_-)=[1-\f{c^4}{4}, +\infty)$ and there is the spectral gap property, i.e. there exists $\de>0$, so that
  		$$
  		\cm_-|_{\{\phi\}^\perp}\geq \de. 
  		$$
  		That is,  for all $h\perp U$, 
  		$\dpr{\cm_- h}{h}\geq \de \|h\|_{L^2}^2$. 
  		\item $\cm_+$ has exactly one negative eigenvalue, $\cm_+[U']=0$ and $ \si_{a.c.}(\cm_+)=[1-\f{c^4}{4}, +\infty)$. 
  		\item   $\cm_+$  satisfies the spectral gap property. Specifically, if  $\Psi_0$ is the eigenfunction, corresponding to the unique negative eigenvalue, then 
  		$$
  		\cm_+|_{\{\Psi_0, \ker[\cm_+]\}^\perp}\geq \de. 
  		$$
  	\end{itemize}
 \end{proposition}
{\bf Remark:} By setting the Lagrange multiplier $\kappa^{-1}=\gamma$, which invariably depends on parameters $c$ and $\la$, and with the assignment 
$\phi:= U$, we have that $\phi$ satisfies \eqref{20b} and also $\cl_\pm=\cm_\pm$ and hence, all spectral properties of $\cl_\pm$ stated in Theorem \ref{theo:10b} follow from the ones presented in Proposition \ref{prop:20b}. 
\subsubsection{The derivation of the Euler-Lagrange equation}
\begin{proof} 
	Let $U\in H^{2}(\mathbf{R})$ be the maximizer of $I[\cdot]$ subject
	to the constraint, $L[u]=\lambda$, so that $I[U]=M_{\lambda}$. 
	Let $\epsilon>0$ be small and $h$ be a real valued  test
		function. Consider 
		$$
		u_{\epsilon}:=\sqrt{\lambda}\frac{U+\epsilon h}{\left(L[U+\epsilon h]\right)^{1/2}},
		$$
		Clearly, $u_{\epsilon}$ satisfies the constraint equation $L[u_{\epsilon}]=\lambda$. 
		We expand $I[u_{\epsilon}]$ and ignore terms of order $O(\epsilon^{3})$ or higher. 
		We have, 
		$$
		I[u_{\epsilon}]=\int	_\rone G\left[\f{\la}{L[U+\epsilon h]} |U+\epsilon h|^2\right] \ dx.
		$$
		 
		Expanding out the terms, we get, 
		\begin{eqnarray*}
		& & 	L[U+\epsilon h] =\int(U''+\epsilon h'')^{2}-c^{2}(U'+\epsilon h')^{2}+(U+\epsilon h)^{2}\ dx \\ 
		& &  =\left(\int_{\mathbf{R}}(U'')^{2}-c^{2}(U')^{2}+U^2\ dx\right)+2\epsilon\left(\int_{\mathbf{R}}(U''h'')-c^{2}(U'h')+U h\ dx\right) \\  & & +\epsilon^{2}\left(\int_{\mathbf{R}}(h'')^{2}-c^{2}(h')^{2}+h^{2}\ dx\right)   =\lambda\left(1+2\epsilon\frac{A}{\lambda}+\epsilon^{2}\frac{B}{\lambda}\right), 
	\end{eqnarray*}
where, for readability we have used notations, 
\begin{eqnarray*}
	A &=& \int_{\mathbf{R}}U''h''-c^{2} U'h'+U h\ dx,   = \langle (\p_{xxxx}+c^2\p_{xx} +1)U, h\rangle\\
	B &=& \int_{\mathbf{R}}(h'')^{2}-c^{2}(h')^{2}+h^{2} \ dx=\langle (\p_{xxxx}+c^2\p_{xx} +1)h, h\rangle .	
\end{eqnarray*}
Therefore, 
\begin{eqnarray*}
	\left(L[U+\epsilon h]\right)^{-1} & = &\lambda^{-1}\left(1+2\epsilon\frac{A}{\lambda}+\epsilon^{2}\frac{B}{\lambda}\right)^{-1}\nonumber \\
	& =& \lambda^{-1}\left[1-2\epsilon\frac{A}{\lambda}+\epsilon^{2}\left(4\frac{A^{2}}{\lambda^{2}}-\frac{B}{\lambda}\right)+O(\epsilon^{3})\right].
\end{eqnarray*}
Combining this with the expansion for $|U+\epsilon h|^2$, we get 
\begin{align*}
|u_\epsilon|^2=\f{\la}{L[U+\epsilon h]} |U+\epsilon h|^2 & =\left[1-2\epsilon\frac{A}{\lambda}+\epsilon^{2}\left(4\frac{A^{2}}{\lambda^{2}}-\frac{B}{\lambda}\right)\right] \left[ |U|^2 + 2\epsilon U h + \epsilon^2 h^2 \right]\\ 
& = \ \underbrace{|U|^2}_{a} + \epsilon \underbrace{\left[ -2\f{A}{\la}|U|^2+2hU\right]}_b \\
& \qquad \qquad +  \epsilon^2 \underbrace{\left[ -4 \f{A}{\la} U	 h + 4\f{A^2}{\la^2} |U	|^2 -\f{B}{\la}|U	|^2 + h^2    \right]}_c .
\end{align*}
As $G(\cdot)$ is a polynomial, we note the expansion, 
		 \begin{align*}
		 G[a+\epsilon b +\epsilon^2 c] & = G[a] +\epsilon G'[a] b +\epsilon^2\left( G'[a]c +\f{G''[a]}{2}b^2\right).
 \end{align*}

Using the terms $a,b$ and $c$ as above, and the relations $G'=F, G'' = F'$, we get the expansion for $I[u_\epsilon]$ as, 
\begin{align}
I[u_\epsilon] & = \int_\rone G[|u_\epsilon|^2] \ dx \label{390b}  \\ 
& = \int_\rone G[|U^2|] \ dx + \epsilon\int	F[|U|^2]\left[ -2\f{A}{\la}|U|^2+2hU\right] \ dx \notag \\
& \qquad\qquad + \epsilon^2\left( \int	 F[|U|^2] \left[ -4 \f{A}{\la} U	 h + 4\f{A^2}{\la^2} |U	|^2 -\f{B}{\la}|U	|^2 + h^2    \right]\right. \notag \\
& \qquad \qquad \qquad \qquad \qquad\qquad \qquad\qquad  + \left. \f{1}{2}\int   F'[|U|^2] \left[ -2\f{A}{\la}|U|^2+2hU\right]^2 \ dx \right)\notag
\end{align}

Using the fact that $\int_\rone G[|U|^2] \ dx = M_\lambda$ and $I[u_\epsilon]\leq M_\lambda $ for $|\epsilon|<< 1$ implies that the term of order $\epsilon$ has to be zero, i.e.

\begin{align}
\int	F[|U|^2]\left[ -2\f{A}{\la}|U|^2+2hU\right] & =A\f{2}{\la}\int F[|U|^2]|U|^2 \ dx -2 \int	F[|U|^2] U h \ dx \label{3100b} \\
& = 2\kappa\int_{\rone}U''h''-c^{2} U'h'+U h \ dx -2\int_\rone F[|U|^2] U h \ dx =0,\notag 
\end{align}
where $\kappa$ is the non negative constant, $\kappa = \f{1}{\la}\int_\rone F[|U|^2]|U|^2 \ dx $.

After factoring $\kappa$, this is equivalent to 

\begin{equation} \label{3250b}
\langle (\p_{xxxx}+c^2\p_{xx} +1-\kappa^{-1} F[|U|^2])U, h \rangle =\langle \cm_-U,h\rangle=0.
\end{equation}
Since this is true for arbitrary $h$, we have thus shown that the maximizer $U$, is a distributional solution of \eqref{80b}.

Recall that $U\in H^2(\rone)$ by construction. Note that as one can recast \eqref{80b} in the form 
$$
U=\f{\la}{M_\la}(-\p_{xxxx}+c^2\p_{xx}+1)^{-1} U^3,
$$
where the Green's function operator $(-\p_{xxxx}+c^2\p_{xx}+1)^{-1}$ is given by the multiplier $\f{1}{\xi^4-c^2 \xi^2+1}$. Note that since $c\in (0, \sqrt{2})$, the multiplier $\f{1}{\xi^4-c^2 \xi^2+1}$ is bounded and we have that $(-\p_{xxxx}+c^2\p_{xx}+1)^{-1}: H^s\to H^{s+4}$ for any $s>0$. Thus one can bootstrap the smoothness of $U$ to any given $s$. Indeed, estimating for $s=0$, 
$$
\|(-\p_{xxxx}+c^2\p_{xx}+1)^{-1} U^p\|_{H^4}\leq C \|U^p\|_{L^2}= C \|U\|_{L^{2p}}^p \leq C \|U\|_{H^2}^p, 
$$
whence $U\in H^4(\rone)$. Iterating on this bound, we have $U,U',U'',U''' \in L^\infty	$, and thus $U^p\in H^4(\rone)$, and inverting the operator, 
$$
\|U\|_{H^8} = \|(-\p_{xxxx}+c^2\p_{xx}+1)^{-1} U^p\|_{H^8}\leq \|U^p\|_{H^4}
$$ So, $U\in H^\infty(\rone)$.

Regarding the linearized operators, we have already established int \eqref{3250b} that
 $\dpr{\cm_-U}{h}=0$ for all real-valued test functions $h$, whence $\cm_- U=0$.

As $U$ is the maximizer of the functional $I[\cdot]$, the second derivative condition necessitates the term of order $\epsilon^2$, in the expansion \ref{390b}, be non-positive or equivalently, 

\begin{align}\label{3110b}
\begin{split}
0 & \leq  \int	 F[|U|^2] \left[ 4 \f{A}{\la} U	 h - 4\f{A^2}{\la^2} |U	|^2 +\f{B}{\la}|U	|^2 - h^2    \right]\ dx\\
& \qquad \qquad \qquad  - \f{1}{2}\int   F'[|U|^2] \left[ -2\f{A}{\la}|U|^2+2hU\right]^2 \ dx   \\
& = \int	 F[|U|^2] \left[ \f{B}{\la}|U	|^2 - h^2    \right] \ dx - \int  2 F'[|U|^2]  \left[ -\f{A}{\la}|U|^2+hU\right]^2 \ dx, 
\end{split}
\end{align}

where the last equality follows from \eqref{3100b}. Owing to the non negativity of the second integral in \ref{3110b}, it follows that, 

$$
0\leq \int	 F[|U|^2] \left[ \f{B}{\la}|U	|^2 - h^2    \right] \ dx= \kappa \left(B - \kappa^{-1}\int	F[|U|^2] h^2 \ dx\right),
$$
which, by the definition of $B$, reasons that
$$
0\leq \langle \cm_- h,h\rangle,
$$ whereby establishing the non-negativity of the operator $\cm_-.$

\subsubsection{$\cm_-$ has spectral gap at zero} 
We now show that $\cm_-$ has spectral gap property at zero. Observe that inequality \eqref{3110b}, in fact, tells us that 
\begin{equation}\label{3500b}
\int  2 F'[|U|^2]U^2  \left[ -\f{A}{\la}|U|+h\right]^2 \ dx\leq \langle \cm_- h,h\rangle.
\end{equation}

 We now claim that $\cm_- |_{\{U\}^\perp}\geq \de>0.$ Assuming the contrary for a contradiction, there exist a sequence $h_n: h_n\perp U, \|h_n\|_{L^2}=1$, such that $\lim_{n\to \infty}\dpr{\cm_- h_n}{h_n}=0$. Then inequality \eqref{3500b} implies that 
 \begin{equation}\label{3160b}
 \lim_{n\to \infty} \int  2 F'[|U|^2]U^2  \left[ -\mu_n U+h_n\right]^2 \ dx = 0,
 \end{equation}
 
 where $\mu_n =\f{A_n}{\la}=\f{\kappa^{-1}\int F[U^2]Uh_n\ dx}{\la} $. Observe that the sequence $\{\mu_n\}$ is bounded as 
 \begin{equation}
 |\mu|\leq \f{\kappa^{-1}}{\la} F[\|U\|_\infty^2]\int |Uh_n|\ dx \leq C \|U\|_2\|h_n\|_2 <C.
 \end{equation}
 So, choose a subsequence such that, $\mu_n \to \mu_0$. 
 
 Suppose $\mu_0 =0$, then \eqref{3160b} implies that, 
 \begin{align*}
 \lim_n \int 2F'[U^2]U^2h_n^2 \ dx = 0,
 \end{align*}
 which further implies, by nature of the function $F$, that 
 \begin{equation}
\lim_n \int F[U^2]h_n^2 \ dx =0. 
 \end{equation}

 But then, 
 \begin{align*}
 0=\lim_n \dpr{\cm_- h_n}{h_n}  & = \ \lim_n \int (h_n'')^2-c^2 (h_n')^2+h_n^2 - \kappa^{-1} \int F[U^2]h_n^2 \\
 & = \ \lim_n \int (h_n'')^2-c^2 (h_n')^2+h_n^2 \ dx.
 \end{align*}
So, it follows that $\lim_n L[h_n]=0$, whence $\lim_n \|h_n\|_{L^2}=0$, a contradiction. 

If $\mu_0\neq 0$, then consider the open set $\ca=\{x\in \rone: U(x)\neq 0\}$, so that $ (F'[|U|^2]U^2)(x)\neq 0$ for $x\in \ca$. Note that since  $U$ is a smooth solution to a regular differential equation, the set $\ca^c$ cannot have points of accumulation, except at infinity. For every test function, $q\in C^\infty_0(\ca)$, we have that $F'[|U|^2]U^2|_{\text{supp }\{ q\}}\geq C_q$, so 
 $$
 |\dpr{q}{h_n-\mu_0 U}|^2\leq \|q\|^2_{L^2} \int_{\text{supp }\{ q\}} |h_n-\mu_0 U|^2 \leq \f{\|q\|_{L^2}^2}{C_q} \int F'[U^2]U^2 |h_n-\mu_0 U|^2 dx,
 $$
 whence $\lim_n \dpr{q}{h_n-\mu_0 U}=0$ for all $q\in C^\infty_0(\ca)$. Let now $\eps>0$ be arbitrary. Then, we can find $q_\eps \in C^\infty_0(\ca):  \|U-q_\eps\|_{L^2}<\eps$. Then, since $h_n\perp U$, 
 \begin{eqnarray*}
 	 |\mu_0| \|U\|_{L^2}^2 &=& |\dpr{U}{h_n-\mu_0 U}|\leq |\dpr{q_\eps}{h_n-\mu_0 U}|+|\dpr{U-q_\eps}{h_n-\mu_0 U}|\leq \\
 	 &\leq &  |\dpr{q_\eps}{h_n-\mu_0 U}|+\eps\|h_n-\mu_0 U\|_{L^2}\leq |\dpr{q_\eps}{h_n-\mu_0 U}|+\eps(1+|\mu_0| \|U\|_{L^2}). 
 \end{eqnarray*}
Taking $\lim_n $ in the estimate above leads to 
$$
|\mu_0| \|U\|_{L^2}^2\leq \eps(1+|\mu_0| \|U\|_{L^2}),
$$
which is contradictory, as $\mu_0\neq 0$ and $\eps$ is arbitrary. 

\subsubsection{Properties of $\cm_+$}
Upon a closer examination of the second integral in \ref{3110b}, the inequality reads, 
$$
0\leq \int	 F[|U|^2] \left[ \f{B}{\la}|U	|^2 - h^2    \right] \ dx - \int 2 F'[|U|^2] \left[U^2h^2 +\f{A^2}{\la^2}|U|^4 -2\frac{A}{\la}|U|^2U h\right]\ dx.
$$ Factoring out $\kappa$, this is equivalent to 
\begin{align}\label{3120b}
0 & \leq \langle \cm_+ h,h \rangle + \kappa^{-1}	\f{A}{\la} \int	 2F'[|U|^2]|U|^2Uh \ dx - \kappa^{-1} \f{A^2}{\la^2} \int 2F'[|U|^2]|U|^4 \ dx \nonumber\\
& \leq \langle \cm_+ h,h \rangle + \kappa^{-1}	\f{A}{\la} \int	 2F'[|U|^2]|U|^2Uh \ dx
\end{align}

Let us note that \eqref{3120b} already implies that $\cm_+$ has at most one negative eigenvalue. Indeed, restricting  $h\perp F'[U^2]U^2U$, we obtain 
$\dpr{\cm_+ h}{h}\geq 0$, which by Rayleigh-Ritz min-max theorem
implies that $n(\cm_+)\leq 1$. On the other hand, testing with $U$, we obtain from the Euler-Lagrange equation, 
$$
\dpr{\cm_+ U}{U}= - 2\kappa^{-1} \int F'[|U|^2]U^4 dx <0, 
$$
whence $n(\cm_+)=1$. 

\end{proof}

 \section{Proof of Theorem \ref{theo:20b}}
 We see that the relevant  eigenvalue problem is in the form 
 $$
 \cj \cl \begin{pmatrix}
 	f \\ g
 \end{pmatrix}=\la \begin{pmatrix}
 f \\ g
\end{pmatrix},
 $$
 where 
 $$
 \cj= \begin{pmatrix}0 & 1\\
 	-1 & 2c\partial_x
 \end{pmatrix},\ \  \cl= \begin{pmatrix}\mathcal{L}_{+} & 0\\
 	0 & 1
 \end{pmatrix}.
 $$
 Clearly, this fits with Corollary \ref{cor:po} as $\cj^*=-\cj$, $\cl^*=\cl$. Also, note that from Theorem \ref{theo:10b}, we have that the Morse index 
 $$
 n^-(\cl)=n^-(\cl_+)+n(Id)=n(\cl_+)=1. 
 $$
 Also, 
 $$
 \ker(\cl)=\begin{pmatrix}
 	\ker(\cl_+)  \\ 0
 \end{pmatrix}=\spn \left[ \begin{pmatrix}
 \phi' \\ 0
\end{pmatrix} \right],
 $$
 by assumption. We now find the elements of the generalized kernel, as prescribed by the instability index theory. Specifically, we look for solution $\begin{pmatrix}
 	f \\ g
 \end{pmatrix} $ of 
$$
\cj \cl  \begin{pmatrix}
	f \\ g
\end{pmatrix} = \begin{pmatrix}
\phi'\\ 0
\end{pmatrix} 
$$
 As it is easy to see $\cj$ is invertible, and in fact $\cj^{-1}= \begin{pmatrix}2c\p_x  & -1\\
 	1 & 0
 \end{pmatrix}$, we need to solve 
$$
  \cl  \begin{pmatrix}
	f \\ g
\end{pmatrix} = \cj^{-1} \begin{pmatrix}
	\phi'\\ 0
\end{pmatrix}  = \begin{pmatrix}
2c \phi''\\ \phi'
\end{pmatrix}. 
$$
 Immediately, this reduces to $g=\phi'$, $\cl_+ f= 2c\phi''$. If $\ker(\cl_+)=\spn[\phi']$, this is uniquely solvable in $\{\phi'\}^\perp$ by Fredholm theory, so $f=2c \cl_+^{-1}[\phi'']\in \{\phi'\}^\perp$. One should also consider the possibility for further elements of the generalized kernel. This amounts to solving the following 
 $$
 \ch  \begin{pmatrix}
 	z_1 \\ z_2
 \end{pmatrix} = \cj^{-1} \begin{pmatrix}
 2c \cl_+^{-1}[\phi'']\\ \phi'
\end{pmatrix} = \begin{pmatrix}
4c^2  \p_x \cl_+^{-1} [\phi'']-\phi'\\  2c   \cl_+^{-1} [\phi'']
\end{pmatrix}. 
 $$
 The second equation is easily solvable, but the first one requires Fredholm compatibility condition, namely right hand side,  $4c^2  \p_x \cl_+^{-1} [\phi'']-\phi'$,  needs to be orthogonal to $\ker(\cl_+)=\spn[\phi']$. This means that  further elements of the generalized kernel exist if and only if 
 $$
 0=\dpr{4c^2  \p_x \cl_+^{-1} \phi''-\phi'}{\phi'}=-(4c^2 \dpr{\cl_+^{-1} \phi''}{\phi''}+ \|\phi'\|^2)
 $$
 Thus, no further elements of the generalized kernels exist, provided 
 $$
 4c^2 \dpr{\cl_+^{-1} \phi''}{\phi''}+ \|\phi'\|^2\neq 0.
 $$
  Finally, we need to verify the sign of the quantity of the matrix $\cd$. We need to compute 
 \begin{align*}
  \cd_{11}  =\dpr{\cl \begin{pmatrix}
  		2c \cl_+^{-1} \phi'\\ \phi'
  	\end{pmatrix}}{\begin{pmatrix}
  	2c \cl_+^{-1} \phi'\\ \phi'
  \end{pmatrix} } & =\dpr{\begin{pmatrix}
  2c \phi''\\ \phi'
\end{pmatrix}}{\begin{pmatrix}
2c \cl_+^{-1} \phi''\\ \phi'
\end{pmatrix}}\\
&=\  4c^2 \dpr{\cl_+^{-1} \phi''}{\phi''}+ \|\phi'\|^2.
  \end{align*}
  Thus, as $n(\ch)=1$, stability is indeed equivalent to $n(\cd)=1$, which is $4c^2 \dpr{\cl_+^{-1} \phi''}{\phi''}+ \|\phi'\|^2<0$. We have also shown  that as long as this quantity is negative, no further element of the generalized kernel are possible, so the index count is complete and it yields $k_{\text{Ham.}}=0$. 
  
  Note that if the condition $4c^2 \dpr{\cl_+^{-1} \phi''}{\phi''}+ \|\phi'\|^2=0$ is met, even though further elements of the generalized kernels are possible, we do have that $k_{\text{Ham.}}=n^-(\cl)-n^{\leq 0}(\cd)=1-n(\cd)\leq 1-n(\cd_{11})=0$, whence stability follows.


 \section{Exponential decay of the wave $\phi$:  Proof of Theorem \ref{theo:30b}}
 Recall that the wave $\phi\in H^\infty(\rone)$ satisfies \eqref{20b}. In order to extract its  asymptotics for large $x$, we shall need to localize in space. 
 
 \subsection{Preliminary steps}
 
 More precisely, let $\psi$ be a smooth cutoff function such that $\psi(x)=0$ for
 $|x|<1/2$ and $\psi(x)=1, |x|>1$, with $\|\psi^{(k)}\|_{\infty}<C$ for
 $k\in\{1,2,3,4\}.$ Let $N>>1$ and $\psi_{N}(x):=\psi(\frac{x}{N})$.  Multiplying the profile 
 equation \eqref{20b} with $\psi_{N},$ we obtain, 
 $$
 \phi_{xxxx}\psi_{N}+c^{2}\phi_{xx}\psi_{N}+\phi\psi_{N}=\gamma F[\phi^2]\phi \psi_{N}.
 $$
 Combining $\phi$ together with $\psi_{N}$ generates a few commutator terms. Specifically, we have the equation, 
\begin{equation}
	\label{220b} 
		\left(\phi\psi_{N}\right)_{xxxx}+c^{2}\left(\phi\psi_{N}\right)+\phi\psi_{N} =\gamma F[\phi^{2}]\phi\psi_{N}+E_N(x),
\end{equation}
where the error terms\footnote{Note that in the expression for $E_N$, the cutoffs are always differentiated, which means that its support is inside $\f{N}{2}<|x|<N$. } are collected herein 
\begin{align*}
E_N(x) &:=\left[4\phi_{xxx}(\psi_{N})_{x}+6\phi_{xx}(\psi_{N})_{xx}+4\phi_{x}(\psi_{N})_{xxx}+\phi(\psi_{N})_{xxxx}\right]\\
& \qquad \qquad \qquad \qquad+c^{2}\left[2\phi_{x}(\psi_{N})_{x}+\phi(\psi_{N})_{xx}\right],
\end{align*}

Inverting the operator $\p_{xxxx}+c^{2}\p_{xx}+1$ in \eqref{220b}  leads to 
 $$
 \phi\psi_{N}(x)=\left(\p_{xxxx}+c^{2}\p_{xx}+1\right)^{-1}\left(\gamma F[\phi^{2}]\phi\psi_{N}+E_N\right),
 $$
 or
 \begin{equation}
 	\label{230b} 
 	\phi\psi_{N}=K\ast(\gamma F[\phi^{2}]\phi\psi_{N}+E_N), \hat{K}(\xi)=\f{1}{\sqrt{2\pi}} \frac{1}{\xi^{4}-c^{2}\xi^{2}+1}
 \end{equation}
 We now need two technical lemma, which will allow us to accurately estimate the decay of $\phi$, based on the representation 
 \eqref{230b}. 
 \subsection{Two technical lemmas}
 The first one provides exponential decay estimates for the kernel $K$, along with its derivatives, up to third order. 
 \begin{lemma}(Exponential bounds for $K$ and its derivatives)
 	\label{le:30b} 
 	
 	Let $c\in [0, \sqrt{2})$. Then, $K, K', K'', K'''$ all have  exponential decay. In fact, 
 	\begin{equation}
 		\label{240b} 
 		|K(x)|+|K'(x)|+|K''(x)|+|K'''(x)|\leq C e^{-\f{\sqrt{2-c^2}}{2}|x|}.
 	\end{equation}
 The stated exponential decay is sharp. 
 \end{lemma}
{\bf Remark:} One can in fact compute $K, K', K'', K'''$ explicitly, but we shall not do so here, as we only need the exponential estimates.
 We postpone the proof of Lemma \ref{le:30b} for the Appendix. 
 
 The other lemma that we need is about long range interactions. 
 \begin{lemma}[Estimate for long-range interactions]
 	\label{le:40b} 
 	
 	Let $0<q<1$ and $\{a_N\}_N$ be a positive bounded sequence, so that for every $\epsilon>0$, there exists $N_0=N_0(\eps)$, so that for all $N>N_0$, 
 	\begin{equation}
 		\label{300b}
 		a_N\leq \epsilon \sum_{m=1}^{N/2} q^{-m} a_{N-m}.
 	\end{equation}
 	Then, there exists $M_1$ and $\sigma>0$, so that 
 	\begin{equation}
 		\label{310b}
 		a_N\leq M_1 q^{-N\sigma}.
 	\end{equation}
 \end{lemma}
We postpone the proof of the Lemma for the Appendix.

{\bf Remark:} 
\begin{enumerate}
	\item The point of the lemma is that, given only the boundedness of $\{a_N\}$ and the exponentially decaying dependence  from the neighbors, one can upgrade to  {\it some} exponential bound for the sequence itself.
\item We remark that from the proof, one can get at least $\si\geq \f{1}{8}$. 
\end{enumerate}

\subsection{Exponential decay of $\phi$}
  Since we have already shown that $\phi\in H^\infty(\rone)$, it follows that 
  $$ 
  \lim_{|x|\to \infty} \phi(x)=0.
  $$ 
  Next, introduce the positive bounded sequence $a_{N}:=\sup_{|x|\geq N}|\phi(x)|$. 
  
  Note that $\lim_N a_N=0$.  We would like to show that $a_N$ satisfies the assumptions of Lemma \ref{le:40b}, with $q=e^{-\f{\sqrt{2-c^2}}{2}}<1$,
   and hence we will have shown that it has some exponential decay. Our starting point is \eqref{230b}. It follows that 
   \begin{align}
   	\label{330b} 
   	a_N\leq \sup_x|\phi\psi_{N}(x)|& \leq |K\ast(\gamma F[\phi|^{2}]\phi\psi_{N}+E_N)|\\
   	& \leq  	C e^{-\de|\cdot|}*|\gamma F[\phi|^{2}]\phi|\psi_{N}+ |K\ast E_N(x)|, \nonumber
   \end{align}
where we introduced $\de:=\f{\sqrt{2-c^2}}{2}$, so that $|K(x)|\leq C e^{-\de |x|}$ from Lemma \ref{le:30b}. 

We first estimate the term $e^{-\de|\cdot|}*|\gamma F[\phi^2]\phi|\psi_{N}$. At this juncture, we would like to note the estimate that, as $\|\phi\|_\infty < \infty,$
$$
|F[\phi^2]\phi| \leq \sum_{j=1}^n  b_j|\phi^{2j+1}| \leq M_{\|\phi\|_\infty}|\phi|^3.
$$
 We have 
 \begin{align*}
  e^{-\de|\cdot|}*|F[\phi^2]\phi|\psi_{N}  & \leq M \int e^{-\de|x-y|}|\phi|^3\psi_{N}(y) \ dy\\
 	& \leq  M \sum_{m=1}^{\left\lceil \frac{N}{2}\right\rceil }\int_{\left\lfloor \frac{N}{2}\right\rfloor +(m-1)\leq|y|\leq\left\lfloor \frac{N}{2}\right\rfloor +m} e^{-\de|x-y|} |\phi(y)|^{3}\ dy\\
 	& \qquad \qquad \qquad +M\int_{|y|>N}  e^{-\de|x-y|} |\phi|^{3}\ dy,
 \end{align*}
It follows that 
  \begin{align*}
    e^{-\de|\cdot|}*|F[\phi^2]\phi|\psi_{N} & \leq C\|\phi\|_{\infty} \sum_{m=1}^{\left\lceil \frac{N}{2}\right\rceil }e^{-\delta(\left\lceil N/2\right\rceil -m)}a_{\left\lfloor \frac{N}{2}\right\rfloor +(m-1)}\int_{\left\lfloor \frac{N}{2}\right\rfloor +(m-1)\leq|y|\leq\left\lfloor \frac{N}{2}\right\rfloor +m}|\phi|\ dy  \\
  	& \qquad\qquad+ a_N \int_{|y|>N}   |\phi|^2\ dy,\nonumber \\
  	& \leq C \sup_{|x|>N/2}|\phi(x)|^2 \sum_{m=1}^{\left\lceil \frac{N}{2}\right\rceil }e^{-\delta(\left\lceil N/2\right\rceil -m)}a_{\left\lfloor \frac{N}{2}\right\rfloor +(m-1)} + a_{N} \int_{|y|>N}   |\phi|^2\ dy. 
  \end{align*}
For the terms $|K\ast E_N(x)|$, we employ a different strategy. As we do not have any {\it a priori} bounds on the derivatives of $\phi$, we integrate by parts first, so that only terms with $\phi$ (and no derivatives) appear in the formula for $K\ast E_N(x)$. Specifically, in in the expansion of $K\ast E$, after integrating by parts an adequate
number of times\footnote{Note no more than three derivatives fall on $\phi$}, will be of the form, 
$$
\int K^{(i)}(x-y)\psi_{N}^{(j)}(y) \phi(y) \ dy, 0\leq i\leq 3, i+j=4.
$$
Note that $j\geq 1$, i.e. there is always at least one derivative falling on the cutoff, hence we always pickup $N^{-1}$ from it. 
We thus estimate 
\begin{align*}
	| \int K^{(i)}(x-y)\psi_{N}^{(j)}(y) \phi(y) \ dy| & \leq\sum_{m=1}^{\left\lceil \frac{N}{2}\right\rceil }\int_{\left\lfloor \frac{N}{2}\right\rfloor +(m-1)\leq|y|\leq\left\lfloor \frac{N}{2}\right\rfloor +m}|K^{(i)}(x-y)||\phi||\psi_{N}^{(j)}|\ dy  \\
	& \leq N^{-j}\sum_{m=1}^{\left\lceil \frac{N}{2}\right\rceil }Ce^{-\delta(\left\lceil N/2\right\rceil -m)}a_{\left\lfloor \frac{N}{2}\right\rfloor +(m-1)} .
\end{align*}
Now let $\epsilon>0$ be arbitrary, since $\lim_{|x|\to\infty}\phi(x)=0,$
we may choose an $N_{1},$ such that $|\phi(x)|<\epsilon$ for $|x|>\left\lfloor \frac{N_{1}}{2}\right\rfloor $,  also $\int_{|y|>N}   |\phi|^2\ dy\leq \eps$, if $N>N_2$, 
and $N_{3}: N_{3}^{-1}<\epsilon$. Let $N(\epsilon):=\max\{N_{1},N_{2}, N_3\}$. 

Then, we have 
\begin{eqnarray*}
	e^{-\de|\cdot|} \ast |F[\phi^2]\phi| \psi_{N}(x) &\leq & C \epsilon \sum_{m=1}^{\left\lceil \frac{N}{2}\right\rceil }Ce^{-\delta(\left\lceil N/2\right\rceil -m)}a_{\left\lfloor \frac{N}{2}\right\rfloor +(m-1)}+C_{1}\epsilon a_{N}, \\
 \left|\int K^{(i)}(x-y)\psi_{N}^{(j)}(y) \phi(y) \ dy \right| &\leq &  C\epsilon\sum_{m=1}^{\left\lceil \frac{N}{2}\right\rceil }e^{-\delta(\left\lceil N/2\right\rceil -m)}a_{\left\lfloor \frac{N}{2}\right\rfloor +(m-1)}. 
\end{eqnarray*}
All in all, for all $\eps>0$ (and large enough $N_0(\eps)$, and $N>N_0(\eps)$), by combining the estimates for the terms in \eqref{330b}, we arrive at the bound \eqref{300b} needed in Lemma \ref{le:30b}. Thus, we have exponential decay for $a_N$. Specifically, Lemma \ref{le:30b} implies that $a_N\leq C e^{-\de_1 N}$, for some $\de_1>0$.   Hence,  we have shown 
$$
|\phi(x)|\leq C e^{-\de_1|x|}. 
$$
We now bootstrap this to the decay rate of the kernel $K$. First, from the equation \eqref{20b}, we write the equivalent weak formulation
$$
	\phi= K\ast F[\phi^2]\phi. 
$$
Let $x: |x|>>1$. We have 
$$
|\phi(x)|\leq \int_{-\infty}^\infty  |K(x-y)| |F[\phi^2]\phi| dy \leq C \int_{-\infty}^\infty e^{-\de|x-y|} |\phi(y)|^3 dy.
$$
Using the {\it a priori}  decay bound $|\phi(x)|\leq C e^{-\de_1|x|}$, which we have just established,  we immediately upgrade (assuming $\de\neq 3\de_1$) from \eqref{324b}, 
$$
|\phi(x)|\leq  C \int_{-\infty}^\infty e^{-\de|x-y|} e^{-3 \de_1|x|}  dy\leq e^{-\min(\de, 3\de_1) |x|}.
$$
Continuing in this fashion, starting from a bound $e^{-\de_j|x|}$, we can upgrade\footnote{If it ever happens that $3\de_j=\de$, we can upgrade to $e^{-(\de-\eps) |x|}$, see the remark after \eqref{324b},  for any $\eps>0$ and then at the next step  to $e^{-\de|x|}$} to $e^{-\min(\de, 3\de_j)|x|}$, which is $e^{-\de_{j+1}|x|}$, where $\de_{j+1}=\min(\de, 3\de_j)$. Thus, after finitely many iterations, we will obtain the desired decay rate $e^{-\de|x|}$, which matches the decay of $K$. 

Similar argument applies to the derivatives of $\phi$. Indeed, 
$$
|\phi'(x)|\leq  |K'| \ast |F[\phi^2]\phi|, 
$$
and since by Lemma \ref{le:30b}, $|K'(x)|\leq C e^{-\de|x|}$, we arrive at the same bounds. Higher order derivatives are handled by an induction argument: (some of) the derivatives are put onto the nonlinear terms $\phi^3$, which continues to enjoy  the decay rate $e^{-3 \de|x|}$ by the induction hypothesis, etc.

\section{Fourth-order NLS}

The existence, stability and decay estimates for the traveling wave solution can be extended to solitary waves for a fourth-order NLS of the form \eqref{14b}.

Solitary wave of the form $u(x,t)=e^{-i\omega t}\phi(x)$, $\omega>0$ for a real-valued function $\phi$, satisfies the equation 
\begin{equation}\label{NLS_prof}
\phi_{xxxx} +\mu \phi_{xx} +\omega \phi -\phi^2 \phi = 0.
\end{equation}

Setting $u=e^{i \om t} (\phi+ v)$ and plug it in \eqref{14b}. After splitting in real and imaginary parts, $v=v_1+i v_2$,  and ignoring terms $O(v^2)$ and higher, we get the system

\begin{equation}
	\label{l:10b} 
	\vec{v}_t= \cj \cl \vec{v}, 
\end{equation} 
\begin{eqnarray*}
\cj &=& \left(\begin{array}{cc} 
	0 & -1 \\ 
	 1 & 0 
\end{array}\right), \ \ \cl=\left(\begin{array}{cc} 
\cl_+ & 0 \\ 
0 & \cl_-
\end{array}\right), \\
\cl_- &=& \p_{xxxx}+\mu \p_{xx}+\omega - \phi^2, \\
		\cl_+ &=& \p_{xxxx}+\mu \p_{xx}+\omega -3\phi^2.
\end{eqnarray*} 

Following the proofs for establishing existence and stability of traveling waves, we obtain the results, 
\begin{theorem}[Existence]\label{th:NLSexist}
Let $\mu<2\sqrt{\omega}$. Then the profile equation \eqref{NLS_prof} has a classical solution $\phi$, with $\phi \in H^\infty$ and exponential decay, 
$$
|\phi(x)|\leq C e^{- \f{\sqrt{2\sqrt{\omega}-\mu}}{2} |x|}, \  0<\mu<2\sqrt{\omega}.
$$
\end{theorem}

\begin{theorem}[Spectral Stability]\label{th:NLSstability}
Let $\mu<2\sqrt{\omega}$ and $\phi$ be the Solitary wave in the existence theorem. Assume also that the wave $\phi$ is non-degenerate, that is $\ker(\cl_+)=\spn[\phi']$. Then, the wave $e^{-i\omega t}\phi$ is spectrally stable, if and only if the Vakhitov-Kolokolov type condition holds 
	\begin{equation}
		\label{NLS:10} 
		\langle \cl_+^{-1}\phi,\phi \rangle < 0.
	\end{equation} 
\end{theorem}
\noindent 
{\bf Remark:} Here, again assuming the mapping $\omega\to \phi_\omega$ is Gateaux differentiable, we have 
$$
\langle \cl_+^{-1}\phi,\phi \rangle=-\p_\omega \|\phi_\omega\|^2.
$$
So, we need to check the function $\omega\to   \|\phi_\omega\|^2$. In the intervals where it increases, we have stability according to \eqref{NLS:10}.

\section{Numerical results}

First, we numerically compute the solution to
\begin{equation}\label{phinumerics}
    \alpha \phi_{xxxx} + \mu \phi_{xx} + \omega \phi - \phi^3,
\end{equation}
which is \eqref{20b} when $\alpha = 1$, $\mu = c^2$,  $\omega = 1$, $\gamma = 1$, and $F(x) = x$, and \eqref{NLS_prof} when $\alpha = 1$. To do this, we start with the known solitary wave solution for the second order NLS equation, which corresponds to $\alpha = 0$ in \eqref{phinumerics}. We gradually increase $\alpha$ to $\alpha = 1$, solving for $\phi$ at each step using a Newton conjugate-gradient method (\cite[chapter 7.2.4]{Yang}). Periodic boundary conditions are used for all computations. Once we have this initial solution, we use pseudo-arclength continuation with AUTO \cite{AUTO} to construct solutions for a wide range of parameters. To do this, we rewrite \eqref{phinumerics} as the Hamiltonian first order system 
\begin{equation}
\frac{du}{dx} = J H'(u),
\end{equation}
where
\begin{equation}
    u = (q_1, q_2, p_1, p_2) = \left(\phi, \phi_x, -\phi_{xxx} - \frac{\mu}{\alpha} \phi_x, \phi_{xx} \right),
\end{equation}
$J$ is the standard $4 \times 4$ symplectic matrix, and the Hamiltonian $H$ is given by
\begin{equation}
    H(u) = q_2 p_1 + \frac{1}{2} p_2^2 - \frac{\omega}{2 \alpha}q_1^2 + \frac{1}{4\alpha} q_1^4 + \frac{\mu}{2 \alpha} q_2^2.
\end{equation}
For parameter continuation with AUTO, we use the equation 
\begin{equation}
    \frac{du}{dx} = J H'(u) + \epsilon H'(u),
\end{equation}
where $\epsilon$ is a small, artificial parameter that breaks the Hamiltonian structure (see \cite[Section 3.1]{Champneys1997}). This is necessary since homoclinic orbits and periodic orbits in Hamiltonian systems are generally codimension zero phenomena, i.e., they persist as system parameters are varied; the additional parameter $\epsilon$ converts this into a codimension one problem.

For the beam equation \eqref{20b} with cubic nonlinearity, we obtain traveling wave solutions $\phi$ for $0 < c < 1.404$ using parameter continuation (this is just shy of the upper bound of $c = \sqrt{2}$ from Theorem \ref{theo:30b}). Two representative solutions to \eqref{20b} are shown in Figure \ref{fig:beamsolutions}.  To determine stability of these traveling wave solutions, following Remark 3 after Theorem \ref{theo:20b}, we plot $\|\phi_c'\|_{L^2}^2$ vs. $c$ (Figure \ref{fig:beamstablity}, solid blue line). From the plot, we see that the traveling wave solution is unstable for $0 < c < c^*$ and stable for $c > c^*$, where $c^* \approx 1.35$ is the value of $c$ where the curve attains its maximum. We obtain confirmation of this result by numerically computing the spectrum of the linearization about $\phi$. The eigenvalues $\lambda$ of \eqref{200b} are computed using a standard eigenvalue solver, and the differential operators are approximated using Fourier spectral differentiation matrices. A plot of the maximum real part of $\lambda$ vs. $c$ (Figure \ref{fig:beamstablity}, dotted orange line) shows that $\phi$ is spectrally unstable for $0 < c < c^*$ and stable for $c > c^*$. For $0 < c < c^*$, the linearization about $\phi$ has a pair of real internal mode eigenvalues (inset in Figure \ref{fig:beamsolutions}, left). As $c$ increases, these collide at the origin when $c = c^*$, after which the internal mode eigenvalues become purely imaginary (not shown).

\begin{figure}
    \centering
    \includegraphics[width=0.475\linewidth]{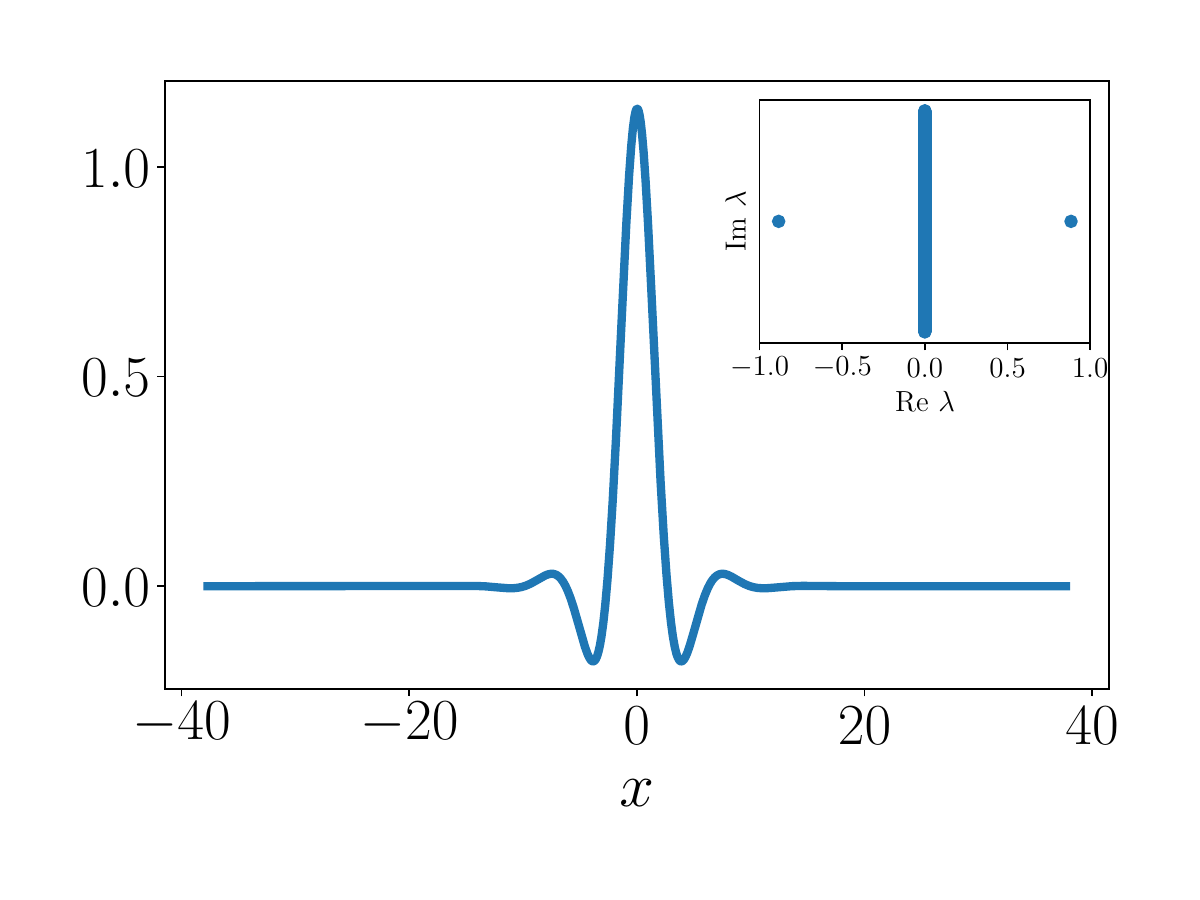}
    \includegraphics[width=0.475\linewidth]{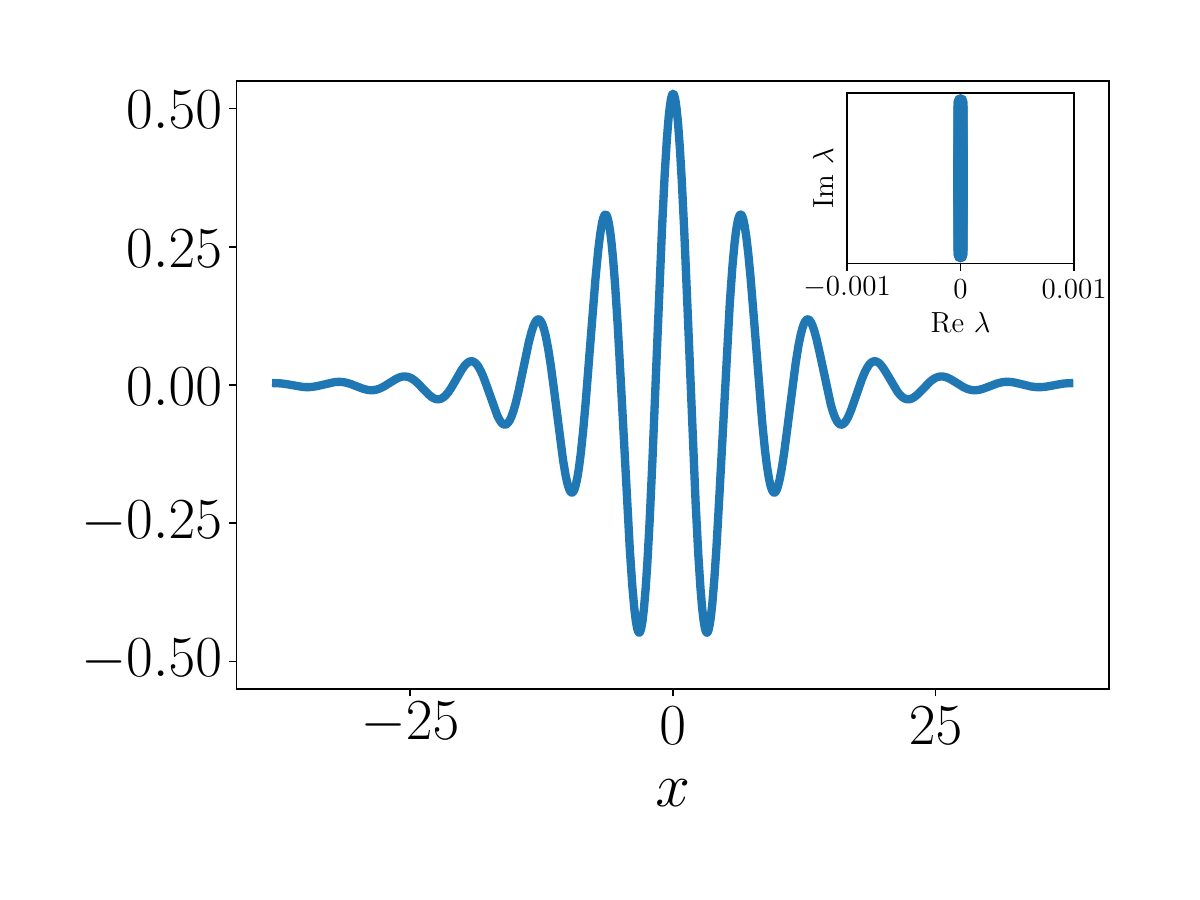}
    \caption{Traveling wave solutions $\phi$ to \eqref{20b} with cubic nonlinearity for $c = 1$ (left) and $c=1.375$ (right). 
    Inset shows spectrum of linearization \eqref{200b} about $\phi$. For $c=1$ (left), there is a pair of real internal mode eigenvalues at 
    $\lambda = \pm 0.885$ (see inset). For $c=1.375$ (right), there is a pair of imaginary internal mode eigenvalues at $\lambda = \pm 0.170 i$ (not shown in inset).
    Parameters: $F(x) = x$, $\gamma = 1$, $x \in [-12\pi, 12\pi]$ with periodic boundary conditions.}
    \label{fig:beamsolutions}
\end{figure}

\begin{figure}
    \centering
    \includegraphics[width=0.475\linewidth]{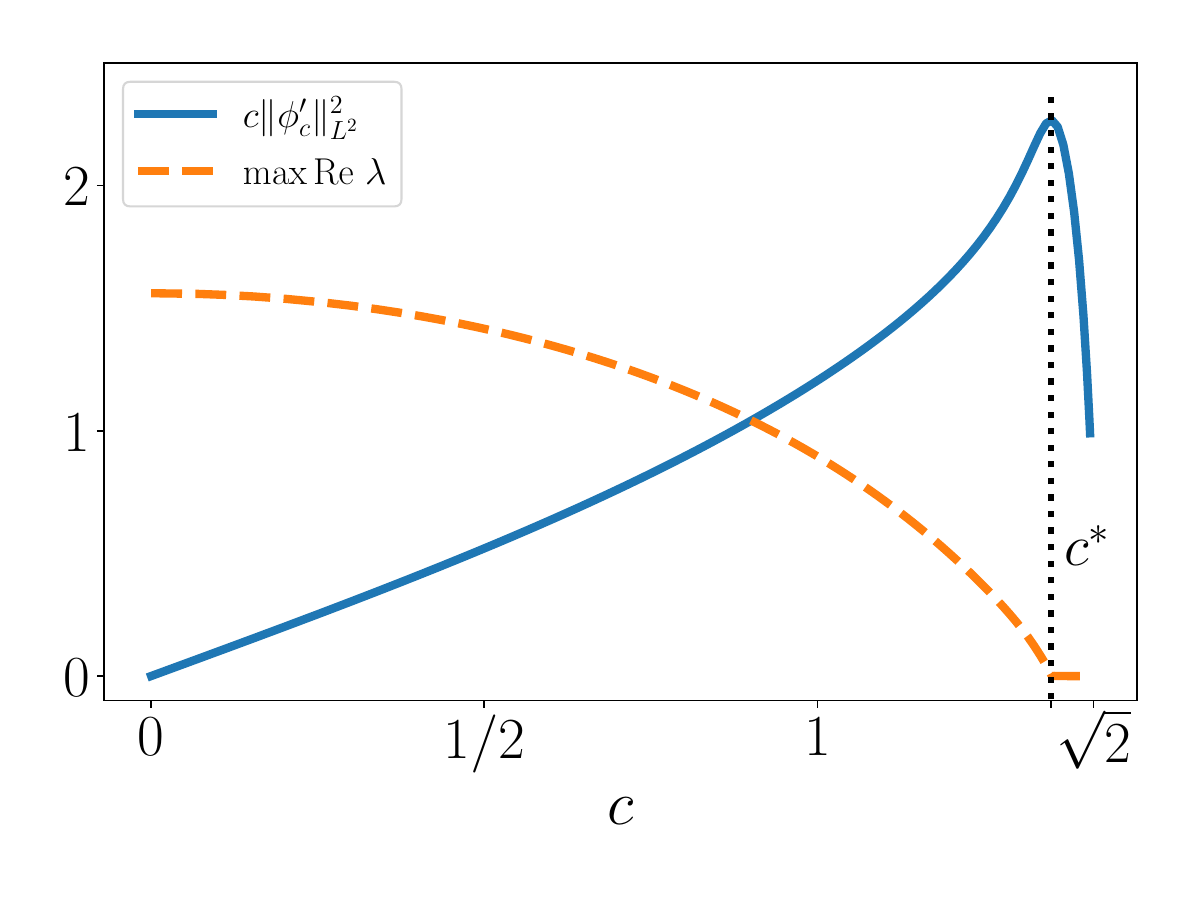}
    \caption{Plot of $c \|\phi_c'\|^2$ (blue solid line) and maximum real part of eigenvalue of linearization \eqref{200b} (orange dotted line) vs. $c$ for traveling wave solutions to \eqref{20b} with cubic nonlinearity. Solution changes from unstable to stable at $c = c^* \approx 1.35$ (black dotted line).    
    Parameters: $\gamma = 1$, $F(x) = x$, $x \in [-12\pi, 12\pi]$ with periodic boundary conditions.}
    \label{fig:beamstablity}
\end{figure}

To characterize the nature of the instability when $c < c^*$, we perturb the traveling wave solution by a small amount in the direction of the unstable eigenfunction (Figure \ref{fig:beamevol}, left). Results from timestepping experiments using this perturbation as an initial condition show that the peak of the traveling wave grows without bound as time evolves (Figure \ref{fig:beamevol}, right). For a timestepping scheme, we use a symplectic and symmetric implicit Runge–Kutta method \cite{Hairer1} (a Python implementation of the \texttt{irk} scheme of order 4 from \cite{Hairer2}) to preserve the symplectic structure of the Hamiltonian system. The Hamiltonian \eqref{eq:beamH} is conserved throughout the simulation.

\begin{figure}
    \centering
    \includegraphics[width=0.475\linewidth]{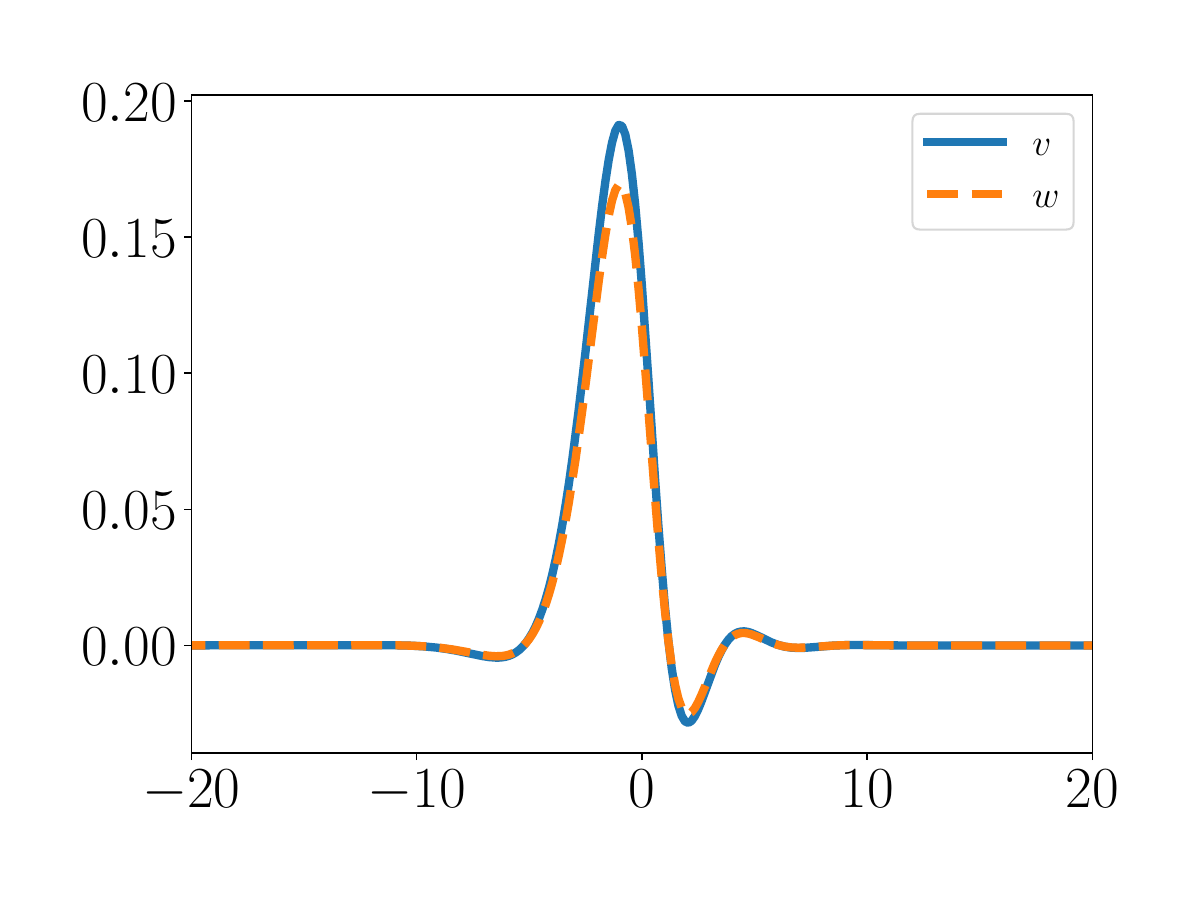}
    \includegraphics[width=0.475\linewidth]{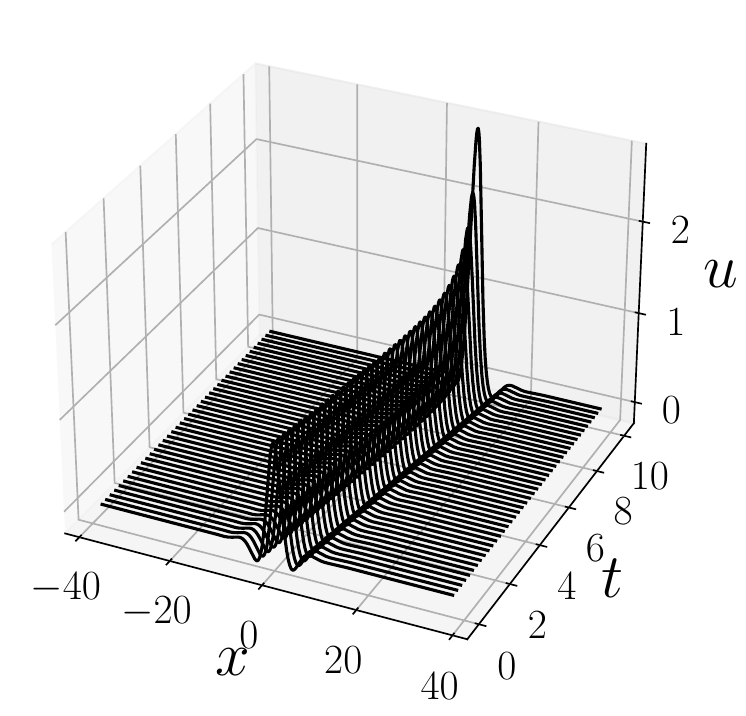}
    \caption{Eigenfunction $(v, w)^\top$ solution to \eqref{200b} corresponding to unstable internal mode eigenvalue $\lambda = 0.885$ for linearization about traveling wave solution to \eqref{20b} with $c = 1$ (left).
    Time evolution of small perturbation of traveling wave solution $\phi$ in direction of unstable eigenfunction (right). Initial condition $(\phi, -c\phi_x)^\top + \epsilon (v, w)^\top$, with $\epsilon = 0.001$.
    Parameters: $F(x) = x$, $\gamma = 1$, $x \in [-12\pi, 12\pi]$ with periodic boundary conditions.}
    \label{fig:beamevol}
\end{figure}


The following equation was introduced in \cite{ChenMcKenna} by Chen and McKenna as a smooth approximation of the model in \cite{McKennaWalter} describing waves propagating on an infinitely long suspended beam:
\begin{equation}\label{Chen}
    u_{tt} + u_{xxxx} + e^u - 1 = 0.
\end{equation}
Traveling wave solutions of the form $u(x, t) = \phi(x - ct)$ satisify the equation
\begin{equation}\label{ChenTravel}
    \phi_{xxxx} + c^2 \phi_{xx} + e^{\phi} - 1 = 0,
\end{equation}
and exist for almost all wavespeeds $c \in (0, \sqrt{2})$ \cite[Theorem 11]{Smets}. Linearization about a traveling wave solution yields the time-independent eigenvalue problem \eqref{200b}, where
\begin{equation}\label{ChenLplus}
    \mathcal{L}_+ = \partial_{xxxx} + c^2 \partial_{xx} + e^{\phi}.
\end{equation}
We first construct a solution $\phi$ to \eqref{ChenTravel} using the string method from \cite{Chamard}. We then continue in the wavespeed parameter $c$ with AUTO to compute solutions for $c \in (0.25, \sqrt{2})$. (We note that the norm of the solution grows significantly as $c$ decreases, hence it was difficult to continue the solution below $c = 0.25$). Since a plot of $c \|\phi_c'\|^2$ vs. $c$ (Figure \ref{fig:NLSstability}, left) is decreasing, this suggests that the VK condition \eqref{20b} is satisfied for all $c$. Furthermore, numerical eigenvalue computation of the linearization about these traveling wave solutions suggests that they are spectrally stable for all $c$. While the results of this paper do not directly apply here, since the nonlinearity $e^u - 1$ in \eqref{Chen} is not of the form in \eqref{eq:main}, these numerical experiments suggest that the theory could be generalized to include this case.

Finally, for the fourth-order NLS equation \eqref{NLS_prof}, we continue in $\omega$ with $\mu$ fixed. For $\mu = 1$, a plot of $\|\phi_\omega\|^2$ vs. $\omega$ (Figure \ref{fig:NLSstability}, right) shows that the solution $\phi$ exists for $\omega > \mu / 4$, in agreement with Theorem \ref{th:NLSexist}. Furthermore, since this function is increasing, it follows from Theorem \ref{th:NLSstability} and the remarks thereafter that this solution is stable.

\begin{figure}
    \centering
    \includegraphics[width=0.475\linewidth]{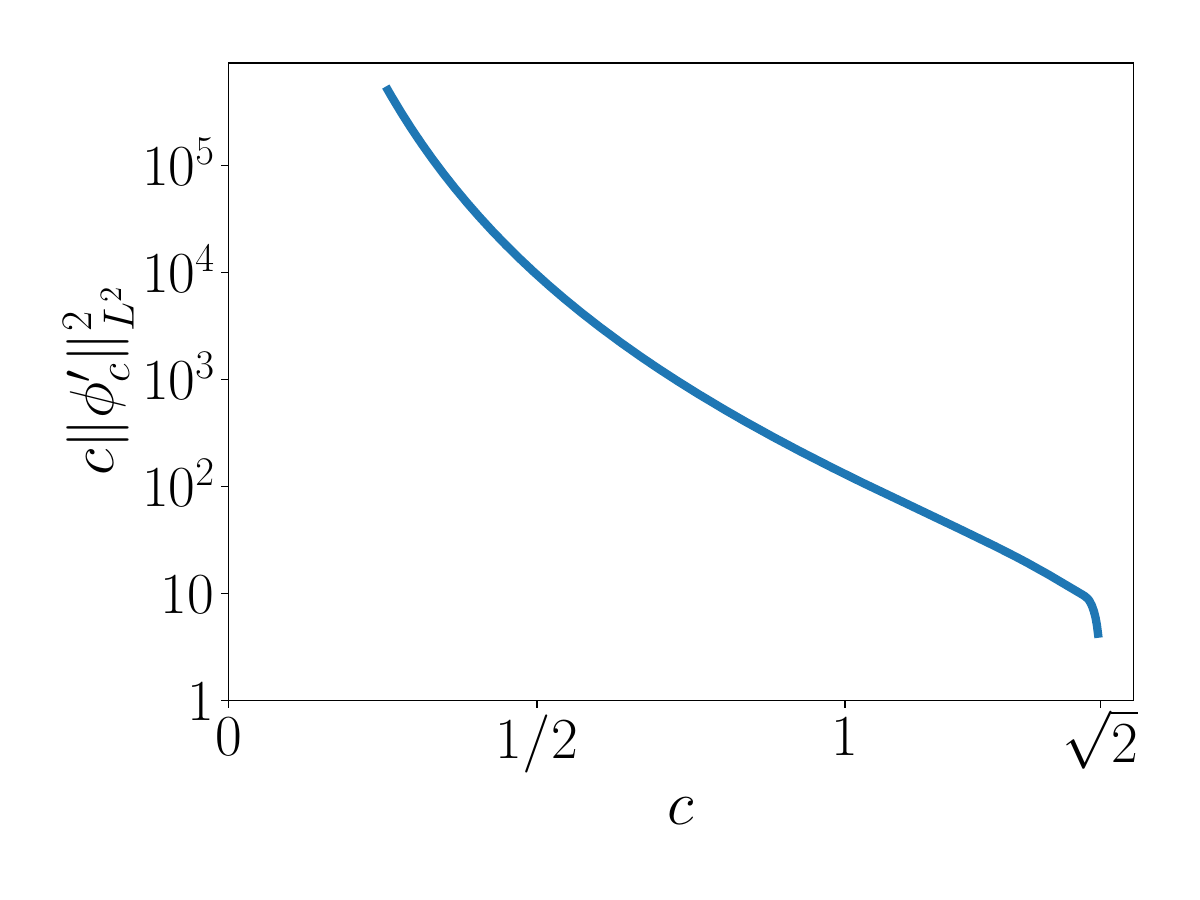}
    \includegraphics[width=0.475\linewidth]{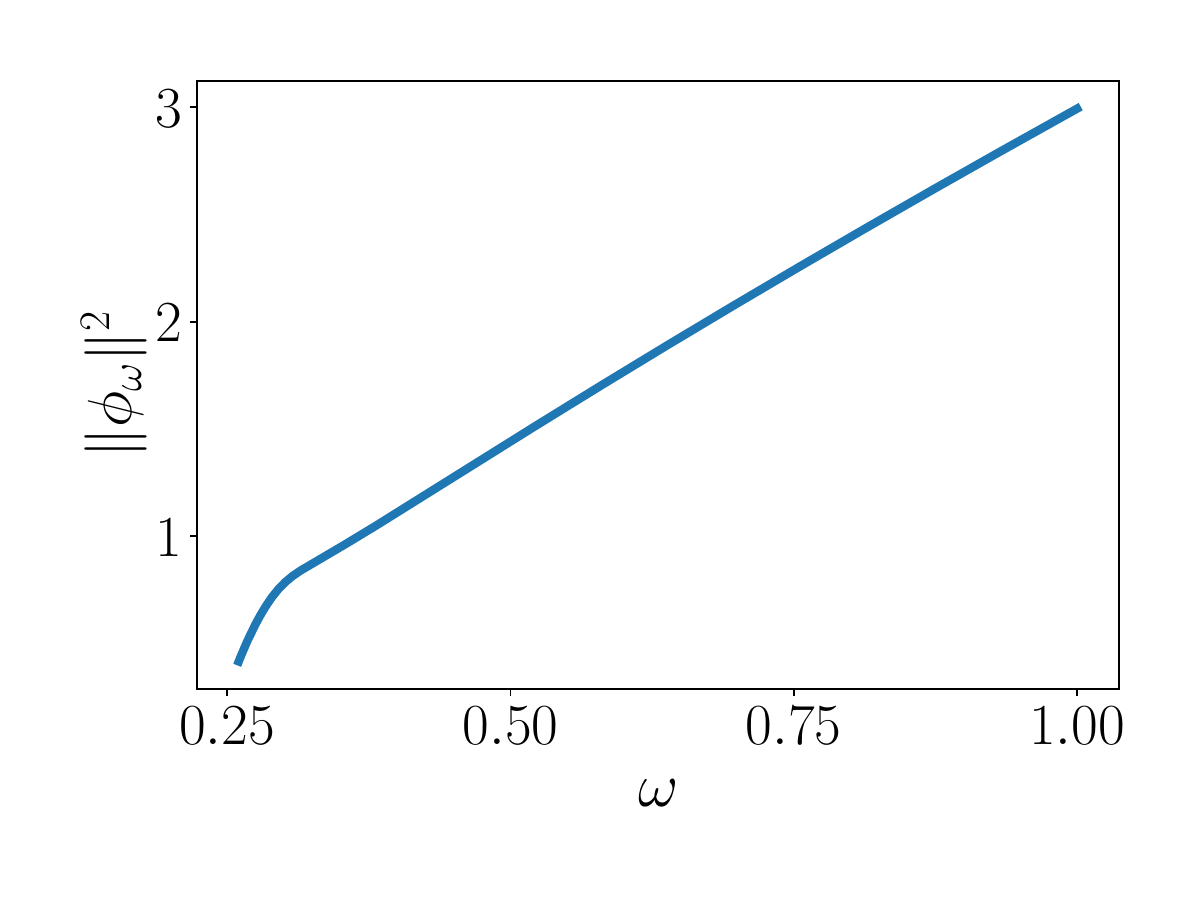}
    \caption{
    Left: semilog plot of $c \|\phi_c'\|^2$ vs. wavespeed $c$ for traveling wave solutions to \eqref{ChenTravel}; parameters are $x \in [-100, 100]$ with periodic boundary conditions.
    Right: plot of $\|\phi_\omega\|^2$ vs. frequency $\omega$ for solutions to \eqref{NLS_prof}; parameters are $\mu = 1$, $x \in [-12\pi, 12\pi]$ with periodic boundary conditions.}
    \label{fig:NLSstability}
\end{figure}

 \appendix
 
 \section{Proof of Lemma \ref{le:30b}} 
 \noindent We need estimates for 
 $$
\int_{-\infty}^\infty  \frac{\xi^je^{i \xi x} }{\xi^{4}-c^{2}\xi^{2}+1} d\xi, j=0,1,2,3.
 $$
 By the partial fraction decompositions, 
 \begin{eqnarray*}
 \frac{1}{\xi^{4}-c^{2}\xi^{2}+1} &=& \frac{\frac{\xi}{2\sqrt{2+c^{2}}}+\frac{1}{2}}{\xi^{2}+\sqrt{2+c^{2}}\xi+1}+\frac{-\frac{\xi}{2\sqrt{2+c^{2}}}+\frac{1}{2}}{\xi^{2}-\sqrt{2+c^{2}}\xi+1}, \\
 	\frac{\xi}{\xi^{4}-c^{2}\xi^{2}+1} &=&  \frac{-\frac{1}{2\sqrt{2+c^{2}}}}{\xi^{2}+\sqrt{2+c^{2}}\xi+1}+\frac{\frac{1}{2\sqrt{2+c^{2}}}}{\xi^{2}-\sqrt{2+c^{2}}\xi+1} \\ 
 	\frac{\xi^{2}}{\xi^{4}-c^{2}\xi^{2}+1} &=& \frac{-\frac{\xi}{2\sqrt{2+c^{2}}}}{\xi^{2}+\sqrt{2+c^{2}}\xi+1}+\frac{\frac{\xi}{2\sqrt{2+c^{2}}}}{\xi^{2}-\sqrt{2+c^{2}}\xi+1}\\
	\frac{\xi^{3}}{\xi^{4}-c^{2}\xi^{2}+1} &=& \frac{\frac{\xi}{2}+\frac{1}{2\sqrt{2+c^{2}}}}{\xi^{2}+\sqrt{2+c^{2}}\xi+1}+\frac{\frac{\xi}{2}-\frac{1}{2\sqrt{2+c^{2}}}}{\xi^{2}-\sqrt{2+c^{2}}\xi+1},
 \end{eqnarray*}
we see that we need estimates of the integrals 
$$
\int_{-\infty}^\infty  \frac{\xi^j e^{i \xi x} }{\xi^{2}\pm \sqrt{2+c^{2}}\xi+1} d\xi, j=0,1.
$$
 By completing the square, matters reduce to bounds for 
 $$
 \int_{-\infty}^\infty \f{\eta^j e^{i \eta x}}{\eta^2+b^2} d\eta, b=\f{\sqrt{2-c^2}}{2}, 
 $$
 which in turn reduces to precise bounds for 
 $$
 \int_{0}^\infty \f{\cos( z x)}{z^2+1} dz,  \int_{0}^\infty \f{z\sin( z x)}{z^2+1} dz.
 $$
 However, the last two integrals are explicit, namely 
 \begin{eqnarray*}
 	 \int_{0}^\infty \f{\cos( z x)}{z^2+1} dz &=&  \f{\pi}{2} e^{-|x|}, \\
 	\int_{0}^\infty \f{z\sin( z x)}{z^2+1} dz&=&  \f{\pi}{2} sgn(x) e^{-|x|},
 \end{eqnarray*}

 
 \section{Proof of Lemma \ref{le:40b}} 
 	Let $M: a_N<M$. 
 Introduce the so-called {\it  envelope sequence}, 
 $$
 b_N:=\sum_{m=0}^{N/2} q^{-m\delta } a_{N-m}.
 $$
 Here $\sum_{m=0}^{A}$ means as long as $0\leq m\leq A$. 
 Clearly, it suffice to show the exponential decay bounds for $b_N$, as $a_N\leq b_N$.
 
  For fixed $\epsilon$, let $N: N>10 N_0(\epsilon)$. Also, fix $\delta\in (0,1)$, say $\delta=1/2$. From \eqref{300b}, we have the bound for every $l: 0<l\leq N/4$, 
  \begin{equation}
  	\label{400b}
  	q^{-l \delta} a_{N-l}\leq \epsilon 	q^{-l \delta} \sum_{j=0}^{(N-l)/2} q^{-j} a_{N-l-j}.
  \end{equation}
 We claim, that within this sum,  either $N-l-j\geq N/2$ or  $j\geq (N-l)/4$. Otherwise $N-l-j< N/2, j< (N-l)/4$ and so 
 $
 N/2-l<j<(N-l)/4,
 $
 whence 
 $
 N<3l<\f{3 N}{4},
 $
 and we reach a contradiction. So, we can estimate in \eqref{400b}, 
 \begin{eqnarray*}
 	q^{-l \delta} a_{N-l}\leq \epsilon 	q^{-l \delta} \sum_{j: N-l-j\geq N/2} q^{-j} a_{N-l-j}+ \sum_{N>j \geq (N-l)/4}  M q^{-j} \\
 	\leq 	\epsilon 	q^{-l \delta} \sum_{j: N-l-j\geq N/2} q^{-j} a_{N-l-j}+ C M q^{-N/8}, 
 \end{eqnarray*}
 as in the second sum $j>\f{3N}{16}>\f{N}{8}$. Thus, using the bound \eqref{400b}
 \begin{eqnarray*}
 	b_N &=& \sum_{l=0}^{N/2} q^{-l\delta } a_{N-l}\leq \sum_{l=0}^{N/4} q^{-l\delta } a_{N-l}+\sum_{l=N/4}^{N/2} q^{-l\delta } a_{N-l}\\
 	&\leq & \sum_{l=0}^{N/4} q^{-\delta l} a_{N-l}+ 
 	M q^{-N\delta /8} \leq \epsilon \sum_{l=0}^{N/4} \sum_{j: N-l-j\geq N/2}  q^{-l \delta} e^{-j} a_{N-l-j}+M q^{-N\delta /8} \\
 	&\leq &  \epsilon \sum_{m=0}^{N/2} q^{-m} a_{N-m} \sum_{l=0}^m q^{(1-\delta)l} +M q^{-N\delta /8} \leq C\epsilon  \sum_{m=0}^{N/2} q^{-m\delta} a_{N-m}+M q^{-N\delta /8}\\
 	&\leq& C\epsilon b_N+M q^{-N\delta /8},
 \end{eqnarray*}
 where we have employed the substitution  $m=j+l$ and $\sum_{l=0}^m q^{(1-\delta)l} \leq C q^{(1-\delta)m}$. As this estimate is possible for every $\epsilon$, choose $\epsilon: C\epsilon<1/2$, and we can clearly turn this into an estimate for $b_N$ by hiding, 
 \begin{equation}
 	\label{405b}
 	b_N\leq C q^{-N\delta/8}.
 \end{equation}


\end{document}